\documentclass[onefignum,onetabnum]{siamart171218}

\usepackage{braket,amsfonts}
\usepackage{yhmath}
\usepackage{amsmath}
\usepackage{mathtools}
\usepackage{mathdots}
\allowdisplaybreaks

\usepackage{array}

\usepackage[caption=false]{subfig}
\usepackage{pgfplots}

\usepackage{mathtools}
\newsiamthm{claim}{Claim}
\newsiamthm{exmp}{Example}
\newsiamremark{remark}{Remark}
\newsiamremark{hypothesis}{Hypothesis}
\crefname{hypothesis}{Hypothesis}{Hypotheses}

\usepackage[T1]{fontenc}
\usepackage[utf8]{inputenc}

\usepackage{booktabs}

\usepackage{algorithmic}

\usepackage{graphicx,epstopdf}

\Crefname{ALC@unique}{Line}{Lines}

\usepackage{amsopn}

\usepackage{xspace}
\usepackage{bold-extra}
\usepackage[most]{tcolorbox}

\colorlet{texcscolor}{blue!50!black}
\colorlet{texemcolor}{red!70!black}
\colorlet{texpreamble}{red!70!black}
\colorlet{codebackground}{black!25!white!25}

\lstdefinestyle{siamlatex}{%
  style=tcblatex,
  texcsstyle=*\color{texcscolor},
  texcsstyle=[2]\color{texemcolor},
  keywordstyle=[2]\color{texemcolor},
  moretexcs={cref,Cref,maketitle,mathcal,text,headers,email,url},
}

\tcbset{%
  colframe=black!75!white!75,
  coltitle=white,
  colback=codebackground, 
  colbacklower=white, 
  fonttitle=\bfseries,
  arc=0pt,outer arc=0pt,
  top=1pt,bottom=1pt,left=1mm,right=1mm,middle=1mm,boxsep=1mm,
  leftrule=0.3mm,rightrule=0.3mm,toprule=0.3mm,bottomrule=0.3mm,
  listing options={style=siamlatex}
}

\newtcblisting[use counter=example]{example}[2][]{%
  title={Example~\thetcbcounter: #2},#1}

\newtcbinputlisting[use counter=example]{\examplefile}[3][]{%
  title={Example~\thetcbcounter: #2},listing file={#3},#1}

\DeclareTotalTCBox{\code}{ v O{} }
{ 
  fontupper=\ttfamily\color{black},
  nobeforeafter,
  tcbox raise base,
  colback=codebackground,colframe=white,
  top=0pt,bottom=0pt,left=0mm,right=0mm,
  leftrule=0pt,rightrule=0pt,toprule=0mm,bottomrule=0mm,
  boxsep=0.5mm,
  #2}{#1}

\patchcmd\newpage{\vfil}{}{}{}
\numberwithin{equation}{section}

\newcommand{\ds}{\mathrm ds}
\newcommand{\dx}{\mathrm dx}
\newcommand{\dtheta}{\mathrm d\theta}

\begin{tcbverbatimwrite}{tmp_\jobname_header.tex}
\title{$H^1$-norm stability and convergence of an L2-type method on nonuniform meshes for subdiffusion equation\thanks{Submitted in June 2022, revised in March 2023, accepted in May 2023. C. Quan is supported by NSFC Grant 12271241, Guangdong Basic and Applied Basic Research Foundation (No. 2023B1515020030), and Shenzhen Science and Technology Program (Grant No. RCYX20210609104358076).}}

\author{
Chaoyu Quan
\thanks{SUSTech International Center for Mathematics, Southern University of Science and Technology, Shenzhen, P.R. China (\email{quanchaoyu@gmail.com}).}
\and
Xu Wu
\thanks{Department of Mathematics,  Harbin Institute of Technology, Harbin 150001, China; Department of Mathematics, Southern University of Science and Technology, Shenzhen, China (\email{11849596@mail.sustech.edu.cn}).}
}

\headers{$H^1$-norm stability and convergence of L2-type method}{C. Quan and X. Wu}
\end{tcbverbatimwrite}
\title{$H^1$-norm stability and convergence of an L2-type method on nonuniform meshes for subdiffusion equation\thanks{Submitted in June 2022, revised in March 2023, accepted in May 2023. C. Quan is supported by NSFC Grant 12271241, Guangdong Basic and Applied Basic Research Foundation (No. 2023B1515020030), and Shenzhen Science and Technology Program (Grant No. RCYX20210609104358076).}}

\author{
Chaoyu Quan
\thanks{SUSTech International Center for Mathematics, Southern University of Science and Technology, Shenzhen, P.R. China (\email{quanchaoyu@gmail.com}).}
\and
Xu Wu
\thanks{Department of Mathematics,  Harbin Institute of Technology, Harbin 150001, China; Department of Mathematics, Southern University of Science and Technology, Shenzhen, China (\email{11849596@mail.sustech.edu.cn}).}
}

\headers{$H^1$-norm stability and convergence of L2-type method}{C. Quan and X. Wu}

\begin{document}
\maketitle
\sloppy

\begin{abstract}
This work establishes $H^1$-norm stability and convergence for an L2 method on general nonuniform meshes when applied to the subdiffusion equation. Under mild constraints on the time step ratio $\rho_k$, such as $0.4573328\leq \rho_k\leq 3.5615528$ for $k\geq 2$, the positive semidefiniteness of a crucial bilinear form associated with the L2 fractional-derivative operator is proved. This result enables us to derive long time $H^1$-stability of L2 schemes. These positive semidefiniteness and $H^1$-stability properties hold for standard graded meshes with grading parameter $1<r\leq 3.2016538$. In addition, error analysis in the $H^1$-norm for general nonuniform meshes is provided, and convergence of order $(5-\alpha)/2$ in $H^1$-norm is proved for modified graded meshes when $r>5/\alpha-1$. To the best of our knowledge, this study is the first work on $H^1$-norm stability and convergence of L2 methods on general nonuniform meshes for the subdiffusion equation.
\end{abstract}



\section{Introduction}
The time-fractional diffusion equation is derived from continuous time random walks \cite{metzler2000random,gorenflo2002time}, which incorporates a fractional derivative in time to model memory effects in diffusing materials.

Over the past decade, numerous numerical methods have been proposed for solving the time-fractional diffusion equation, many of which use uniform time meshes. 
For example, the  L1 scheme of $(2-\alpha)$-order has been extensively developed by Langlands and Henry \cite{langlands2005accuracy}, Sun-Wu \cite{sun2006fully}, and Lin-Xu \cite{lin2007finite}, and others.
Alikhanov \cite{alikhanov2015new} proposes the L2-1$_\sigma$ scheme for the time-fractional diffusion equation with variable coefficients, which has second-order accuracy in time. 
Gao-Sun-Zhang \cite{gao2014new} study an L1-2 method of $(3-\alpha)$-order on uniform meshes, while Lv-Xu \cite{lv2016error} analyze a slightly different L2 fractional-derivative operator for uniform meshes, achieving optimal convergence of $(3-\alpha)$-order in time under strong regularity assumptions on the exact solution. More recently, Alikhanov-Huang \cite{alikhanov2021high} propose an L2 approximation and corresponding schemes for the subdiffusion equation with variable coefficients. 

In recent years, numerical methods on nonuniform time meshes for solving the time-fractional diffusion equation have garnered increasing attention, particularly in the case of graded meshes. 
In fact, the exact solution to the time-fractional diffusion equation can exhibit low regularity near the initial time, which can lead to a reduced convergence rate of numerical solutions.
To overcome this challenge, researchers have explored the use of nonuniform time meshes to achieve the desired sharp convergence rate even under low regularity assumptions on the exact solution.
For example, Stynes-Riordan-Gracia \cite{stynes2017error} prove the sharp error analysis of L1 scheme on graded meshes.
Kopteva  \cite{kopteva2019error} provides a different analysis framework of the L1 scheme on graded meshes in two and three spatial dimensions.
Chen-Stynes \cite{chen2019error} prove the second-order convergence of L2-1$_\sigma$ scheme on fitted meshes combining the graded meshes and quasiuniform meshes.
Kopteva-Meng \cite{kopteva2020error}  provide sharp pointwise-in-time error bounds for quasi-graded termporal meshes with arbitrary degree of grading for L1 and L2-1$_\sigma$ schemes. 
Later, Kopteva \cite{kopteva2021error} generalizes this sharp pointwise error analysis to an L2-type scheme on quasi-graded meshes.
In the case of general nonuniform meshes, Liao-Li-Zhang establish the sharp error analysis for the L1 scheme of linear reaction-subdiffusion equations in \cite{liao2018sharp} where a useful discrete fractional Gr\"onwall inequality is proposed and then Liao-McLean-Zhang \cite{liao2019discrete,liao2018second} further explore the L2-1$_\sigma$ scheme.

In addition to the L1, L2-1$_\sigma$, and L2 methods on nonuniform meshes, it is worth mentioning that convolution quadrature methods with corrections have also shown promise in overcoming the convergence rate problem for the time-fractional diffusion equation, see for example \cite{jin2017correction,jin2020subdiffusion,jin2016two} and the references therein.

In this work, we first investigate the $H^1$-stability of an L2-type method (as studied in \cite{lv2016error,kopteva2021error}) on general nonuniform meshes for subdiffusion equations. We focus on the L2 fractional-derivative operator, denoted by $L_k^\alpha$, and prove that the following bilinear form
\begin{equation}\label{eq:Bn}
    \mathcal B_n(v,w) = \sum_{k=1}^{n}\langle L_k^\alpha v, \delta_k w\rangle,\quad \delta_k w \coloneqq w^k-w^{k-1},~n\geq 1,
\end{equation}
is positive semidefinite, under the mild restrictions outlined in \eqref{thm1cond1}--\eqref{thm1cond2} for the time step ratios $\rho_k$ (see Theorem \ref{thm1} for details).
If $0.4573328\leq \rho_k\leq 3.5615528,$
these mild restrictions are satisfied as shown in Corollary \ref{cor}. 
Of particular note is the fact that the positive semidefiniteness of $\mathcal B_n$ on general nonuniform meshes is unknown for both the L2 and L2-1$_\sigma$ operators, as highlighted in \cite{liao2020second} and \cite[Table 1]{ji2020adaptive}.
This positive semidefiniteness allows us to establish
 long time $H^1$-stability of the implicit L2 scheme for the subdiffusion equation with homogeneous Dirichlet boundary condition (see Theorem \ref{thm2}).
In addition, we consider the special case of graded meshes and demonstrate that if the grading parameter $1<r\leq 3.2016538$, then $\mathcal B_n$ is positive semidefinite, allowing for the establishment of similar $H^1$-stability for the L2 scheme, as outlined in Theorem \ref{thm3}. These findings offer a valuable contribution to the existing literature on numerical methods for subdiffusion equations, providing new insights into the stability properties of L2-type methods on nonuniform meshes.

Building on the positive definiteness results, we next provide the error analysis in $H^1$-norm for the L2 scheme of subdiffusion equation on general nonuniform meshes. 
In particular, we prove that the convergence rate in $H^1$-norm is $\mathcal O(N^{-(5-\alpha)/2})$ for the L2 scheme on modified graded meshes when $r>5/\alpha-1$ (see Theorem \ref{thm:conv_graded}). Although numerical experiments suggest that the actual convergence order in $H^1$-norm could be as high as $3-\alpha$, which is larger than $(5-\alpha)/2$, we are currently unable to provide a proof within our analytical framework. We plan to investigate this matter further in future studies.  
It is worth mentioning that \cite{huang20200optimal,huang2020optimal,huang2022sharp} provide sharp error analysis in the $H^1$-norm for the L1 and L2-1$_\sigma$ schemes. However, establishing $H^1$-norm convergence for the L2 scheme poses a challenge due to the non-monotonicity of the L2 coefficients and the failure of the discrete fractional Grönwall inequality.

Overall, our work contributes to the growing body of research on nonuniform time meshes for time-fractional diffusion equations. By demonstrating the positive semidefiniteness of a bilinear form associated with the L2 scheme on general nonuniform meshes, we establish the $H^1$-stability of this method and provide an $H^1$-norm error analysis for subdiffusion equations. These findings may be useful for designing accurate and efficient numerical algorithms for a wide range of time-fractional problems in mathematical physics and engineering.

This work is organized as follows. 
Section \ref{sect2} provides the derivation, explicit expression, and reformulation of the L2 fractional-derivative operator. In Section \ref{sect3}, we prove the positive semidefiniteness of $\mathcal B_n$ under mild restrictions on the time step ratios. Section \ref{sect4} establishes the long time $H^1$-stability of the L2 scheme based on the positive semidefiniteness result, including a discussion of the case of graded meshes. The $H^1$-norm error estimates of the L2 schemes are provided for both general nonuniform meshes and modified graded meshes in Section \ref{sect5}. Finally, in Section \ref{sect6}, we present numerical tests on the L2 scheme using the modified graded mesh.

\section{Discrete fractional-derivative operator}\label{sect2}
In this part we show the derivation, explicit expression and reformulation  of L2 operator on general nonuniform mesh.

We consider the L2 approximation of the Caputo fractional derivative 
\begin{align*}
    \partial_t^\alpha u = \frac{1}{\Gamma(1-\alpha)} \int_0^t \frac{u'(s)}{(t-s)^\alpha} \, \ds. 
\end{align*}
Take a nonuniform time mesh $0 = t_0<t_1<\ldots<t_{k-1}<t_k<\ldots$ with $k\geq 1$. 
When $k=1$, we use the standard linear Lagrangian polynomial interpolating $\{u^0,u^1\}$:
$
H_1^1 (t) \coloneqq \frac{t-t_1}{t_{0}-t_1} u^0 + \frac{t-t_0}{t_{1}-t_0} u^1.
$ 
When $k\geq 2$, for $1\leq j\leq k-1$, we use the standard quadratic Lagrangian polynomial interpolating $\{u^{j-1},u^j,u^{j+1}\}$:
\begin{equation}\label{eq:H2j}
\begin{aligned}
H_2^j (t) &\coloneqq \frac{(t-t_j)(t-t_{j+1})}{(t_{j-1}-t_j)(t_{j-1}-t_{j+1})} u^{j-1} + \frac{(t-t_{j-1})(t-t_{j+1})}{(t_{j}-t_{j-1})(t_{j}-t_{j+1})} u^{j} \\ 
&\quad + \frac{(t-t_{j-1})(t-t_{j})}{(t_{j+1}-t_{j-1})(t_{j+1}-t_{j})} u^{j+1},
\end{aligned}
\end{equation}
while for $j=k$, we use the quadratic Lagrangian polynomial $H_2^{k-1} (t)$ defined in \eqref{eq:H2j}. Let $\tau_j = t_j-t_{j-1}$. 
At $t = t_k$, the fractional derivative $\partial_t^\alpha u(t)$ is approximated by the discrete fractional-derivative operator
{\footnotesize
\begin{equation}\label{eq:Lk0}
\begin{aligned}
L_1^\alpha u& = \frac{u^1-u^0}{\Gamma(2-\alpha)\tau_1^{\alpha}},\\
    L_k^\alpha u 
    & = \frac{1}{\Gamma(1-\alpha)} \left(
    \sum_{j=1}^{k-1}\int_{t_{j-1}}^{t_j} \frac{\partial_s H_2^j(s)}{(t_k-s)^\alpha}\,\ds + \int_{t_{k-1}}^{t_k} \frac{\partial_s H_2^{k-1}(s)}{(t_k-s)^\alpha}\,\ds \right) \\
    & = \frac{1}{\Gamma(1-\alpha)}\left( \sum_{j=1}^{k-1} ( a_{j}^{(k)} u^{j-1} +  b_{j}^{(k)}u^j +  c_{j}^{(k)} u^{j+1} ) 
    + a_{k}^{(k)} u^{k-2} +  b_{k}^{(k)}  u^{k-1} +  c_{k}^{(k)} u^{k}\right),~k\geq 2
\end{aligned}   
\end{equation}}
where for $1\leq j\leq k-1$,
{\small
\begin{equation}
\begin{aligned}\label{eq:aj}
    a_{j}^{(k)} & =  \int_{t_{j-1}}^{t_j} \frac{2s -t_j-t_{j+1}}{\tau_{j}(\tau_{j}+\tau_{j+1})} \frac{1}{(t_k-s)^\alpha}\,\ds 
   = \int_0^1 \frac{-2 \tau_j(1-\theta)-\tau_{j+1}}{(\tau_{j}+\tau_{j+1})(t_k-(t_{j-1}+\theta \tau_j))^\alpha}\,\dtheta,\\
     b_{j}^{(k)} & = -\int_{t_{j-1}}^{t_j} \frac{2s -t_{j-1}-t_{j+1}}{\tau_{j}\tau_{j+1}} \frac{1}{(t_k-s)^\alpha}\,\ds 
   = -\int_0^1 \frac{2 \tau_j\theta-\tau_j-\tau_{j+1}}{\tau_{j+1}(t_k-(t_{j-1}+\theta \tau_j))^\alpha}\,\dtheta, \\
     c_{j}^{(k)}  & =  \int_{t_{j-1}}^{t_j} \frac{2s -t_{j-1}-t_{j}}{\tau_{j+1}(\tau_{j}+\tau_{j+1})} \frac{1}{(t_k-s)^\alpha}\,\ds 
   = \int_0^1 \frac{\tau_j^2(2\theta-1)}{\tau_{j+1}(\tau_{j}+\tau_{j+1})(t_k-(t_{j-1}+\theta \tau_j))^\alpha}\,\dtheta,
\end{aligned}
\end{equation}}
and
{\small
\begin{align*}
    a_{k}^{(k)} & =  \int_{t_{k-1}}^{t_k} \frac{2s -t_{k-1}-t_{k}}{\tau_{k-1}(\tau_{k-1}+\tau_{k})} \frac{1}{(t_k-s)^\alpha}\ds 
   = \int_0^1 \frac{\tau_k^2(2\theta-1)}{\tau_{k-1}(\tau_{k-1}+\tau_{k})(t_k-(t_{k-1}+\theta \tau_k))^\alpha}\dtheta,\\
  b_{k}^{(k)}& = -\int_{t_{k-1}}^{t_k} \frac{2s -t_{k-2}-t_{k}}{\tau_{k-1}\tau_{k}} \frac{1}{(t_k-s)^\alpha}\,\ds 
   = -\int_0^1 \frac{\tau_k(2\theta-1)+\tau_{k-1}}{\tau_{k-1}(t_k-(t_{k-1}+\theta \tau_k))^\alpha}\,\dtheta, \\
    c_{k}^{(k)} & =  \int_{t_{k-1}}^{t_k} \frac{2s -t_{k-2}-t_{k-1}}{\tau_{k}(\tau_{k-1}+\tau_{k})} \frac{1}{(t_k-s)^\alpha}\,\ds 
   = \int_0^1 \frac{2\tau_k\theta+\tau_{k-1}}{(\tau_{k-1}+\tau_{k})(t_k-(t_{k-1}+\theta \tau_k))^\alpha}\,\dtheta.
\end{align*}}
It can be verified that $ a_{j}^{(k)}<0$, $b_{j}^{(k)}>0$, $ c_{j}^{(k)}>0$ for $1\leq j\leq k-1$, and $a_{k}^{(k)}>0$, $b_{k}^{(k)}<0$, $ c_{k}^{(k)}>0$.
Furthermore, 
$ a_{j}^{(k)}+b_{j}^{(k)}+ c_{j}^{(k)} =0$
always holds for $1\leq j\leq k$.

Specifically speaking, we can figure out the explicit expressions of $ a_{j}^{(k)}$ and $ c_{j}^{(k)}$ as follows (note that $b_j^k = - a_{j}^{(k)}- c_{j}^{(k)}$): for $1\leq j\leq k-1$,
\begin{equation}\label{eq:akj}
\begin{aligned}
   a_{j}^{(k)}
   & = \frac{\tau_{j+1}}{(1-\alpha)\tau_j(\tau_j+\tau_{j+1})}(t_k-t_j)^{1-\alpha}-\frac{2\tau_{j}+\tau_{j+1}}{(1-\alpha)\tau_j(\tau_j+\tau_{j+1})}(t_k-t_{j-1})^{1-\alpha}\\
   &\quad +\frac{2}{(2-\alpha)(1-\alpha)\tau_j(\tau_j+\tau_{j+1})}\left[(t_k-t_{j-1})^{2-\alpha}-(t_k-t_j)^{2-\alpha}\right],\\
    c_{j}^{(k)}
   &= \frac{1}{(1-\alpha)\tau_{j+1}(\tau_j+\tau_{j+1})}
   \Big[-\tau_j( (t_k-t_{j-1})^{1-\alpha}+ (t_k-t_{j})^{1-\alpha})\\
   & \quad+2(2-\alpha)^{-1} ( (t_k-t_{j-1})^{2-\alpha}- (t_k-t_{j})^{2-\alpha}) \Big],
\end{aligned}
\end{equation}
while for $j=k$,
\begin{equation}\label{eq:a_k1}
\begin{aligned}
    a_{k}^{(k)}
    &=\frac{\alpha\tau_k^2}{(2-\alpha)(1-\alpha)\tau_{k-1}(\tau_{k-1}+\tau_{k}) \tau_k^\alpha},\\
   c_{k}^{(k)}
    &=\frac{1}{(1-\alpha)\tau_k^\alpha}+\frac{\alpha\tau_k}{(2-\alpha)(1-\alpha)(\tau_{k-1}+\tau_{k}) \tau_k^\alpha}.
    \end{aligned}
\end{equation}
We reformulate the discrete fractional derivative $L_k^\alpha$ in \eqref{eq:Lk0} as
\begin{equation}\label{eq:Lk1}
\small
\begin{aligned}
L_1^\alpha u& = \frac{1}{\Gamma(2-\alpha)\tau_1^{\alpha}}\delta_1 u,\\
    L_k^\alpha u
  &=\frac{1}{\Gamma(1-\alpha)}\left( ( c_{k}^{(k)} + c_{k-1}^{(k)}) \delta_k u- a_{k}^{(k)}\delta_{k-1} u - a_{1}^{(k)}\delta_1 u  +\sum_{j=2}^{k-1}  d_{j}^{(k)} \delta_j u \right), \quad k\geq 2,
    \end{aligned} 
\end{equation}
where $\delta_j u= u^j-u^{j-1}$ and $ d_{j}^{(k)}\coloneqq c_{j-1}^{(k)}- a_{j}^{(k)}.$
To establish the $H^1$-stability of L2-type method for fractional-order parabolic problem, we shall prove the positive semidefiniteness of $\mathcal B_n$ defined in \eqref{eq:Bn}.

\section{Positive semidefiniteness of bilinear form $\mathcal B_n$}\label{sect3}
\begin{lemma}[Properties of $ a_{j}^{(k)}$, $ c_{j}^{(k)}$ and $ d_{j}^{(k)}$]\label{lemmapro}
Given a nonuniform mesh $\{\tau_j\}_{j\geq 1}$,
the following properties of the L2 coefficients in \eqref{eq:aj} hold:
\begin{itemize}
\item[(P1)] $ a_{j}^{(k)}<0,~1\leq j\leq k-1,~k\geq 2$;
\item[(P2)] $ a_{j}^{(k+1)}- a_{j}^{(k)}>0,~1\leq j\leq k-1,~k\geq 2$; 
\item[(P3)] $ a_{j+1}^{(k)}- a_{j}^{(k)}<0,~1\leq j\leq k-2,~k\geq 3$;
\item[(P4)] $ a_{j+1}^{(k)}- a_{j}^{(k)}< a_{j+1}^{(k+1)}- a_{j}^{(k+1)},~1\leq j\leq k-2,~k\geq 3$;
\item[(P5)] $ c_{j}^{(k)}>0,~1\leq j\leq k-1,~k\geq 2$;
\item[(P6)] $ c_{j}^{(k+1)}- c_{j}^{(k)}<0,~1\leq j\leq k-1,~k\geq 2$; 
\item[(P7)] $ d_{j}^{(k)}>0,~2\leq j\leq k-1,~k\geq 3$;
\item[(P8)] $ d_{j}^{(k+1)}- d_{j}^{(k)}<0,~2\leq j\leq k-1,~k\geq 3$.
\end{itemize}
Furthermore, if the nonuniform mesh $\{\tau_j\}_{j\ge 1}$,
 with  $\rho_j \coloneqq \tau_j/\tau_{j-1}$ satisfies
\begin{equation}\label{condition:rho}
    \frac{1}{\rho_{j+1}}\ge\frac{1}{\rho_{j}^2(1+\rho_{j})}-3,\quad \forall j\geq 2,
\end{equation}
then the following properties of $ d_{j}^{(k)}$ hold:
\begin{itemize}
\item[(P9)] $ d_{j+1}^{(k)}- d_{j}^{(k)}>0,~2\leq j\leq k-2,~k\geq 4$;
\item[(P10)] $ d_{j+1}^{(k)}- d_{j}^{(k)}> d_{j+1}^{(k+1)}- d_{j}^{(k+1)},~2\leq j\leq k-2,~k\geq 4$.
\end{itemize}
\end{lemma}
\begin{proof}
We first provide two equivalent forms of $ a_{j}^{(k)}$ in \eqref{eq:aj} as follows:
\begin{equation}\label{eq:akj_rem}
\small
    \begin{aligned}
         a_{j}^{(k)} 
        &=\frac{1}{\tau_{j}+\tau_{j+1}}\int_0^1 (t_k-(t_{j-1}+s \tau_j))^{-\alpha}\,{\rm d}( \tau_j s^2 -(2\tau_j+\tau_{j+1})s)\\
     &=-(t_k-t_j)^{-\alpha}+\frac{\alpha\tau_j}{\tau_j+\tau_{j+1}}
    \int_0^{1} (\tau_j+\tau_{j+1}+s\tau_j)
    (1-s)(t_k-t_{j}+s\tau_j)^{-\alpha-1}\,\ds
    \end{aligned}
\end{equation}
and
\begin{equation}\label{eq:akj1_rem}
\small
    \begin{aligned}
         a_{j}^{(k)} & =\int_0^1 \frac{-2 \tau_j(1-s)-\tau_{j+1}}{(\tau_{j}+\tau_{j+1})(t_k-(t_{j-1}+s \tau_j))^\alpha}\,\ds =\int_0^1 \frac{-2 \tau_j s-\tau_{j+1}}{(\tau_{j}+\tau_{j+1})(t_k-t_{j}+s \tau_j)^\alpha}\,\ds\\
        &=\frac{1}{\tau_{j}+\tau_{j+1}}\int_0^1 (t_k-t_{j}+s \tau_j)^{-\alpha}\,{\rm d}( - \tau_j s^2-\tau_{j+1}s)\\
    &=-(t_k-t_{j-1})^{-\alpha}-\frac{\alpha\tau_j}{\tau_j+\tau_{j+1}}
    \int_0^{1} (\tau_j+\tau_{j+1}-s\tau_j)(1-s)
    (t_k-t_{j-1}-s\tau_j)^{-\alpha-1}\ds.
    \end{aligned}
\end{equation}
It is not difficult to see $ a_{j}^{(k)}<0$, i.e., (P1) holds. 

Combining \eqref{eq:akj_rem} and \eqref{eq:akj1_rem}, we have
\begin{equation}\label{ineq:akjdiff}
\small
    \begin{aligned}
     a_{j+1}^{(k)}- a_{j}^{(k)}
    &=-\frac{\alpha\tau_j}{\tau_j+\tau_{j+1}}
    \int_0^{1} (\tau_j+\tau_{j+1}+s\tau_j)
    (1-s)(t_k-t_{j}+s\tau_j)^{-\alpha-1}\,\ds
    \\
    &-\frac{\alpha\tau_{j+1}}{\tau_{j+1}+\tau_{j+2}}
     \int_0^{1}(\tau_{j+1}+\tau_{j+2}-s\tau_{j+1}) (1-s)(t_k-t_{j}-s\tau_{j+1})^{-\alpha-1}\ds
     <0,
    \end{aligned}
\end{equation}
where we use the form \eqref{eq:akj_rem}  for $ a_{j}^{(k)}$ and the form \eqref{eq:akj1_rem}  for $ a_{j+1}^{(k)}$.
Therefore (P3) holds.
Moreover, for any fixed $s$,
$
  (t_k-t_{j-1})^{-\alpha},~(t_k-t_{j-1}-s\tau_{j})^{-\alpha-1},~(t_k-t_j+s\tau_{j})^{-\alpha-1} ~ \mbox{and}~(t_k-t_j-s\tau_{j+1})^{-\alpha-1}
$
all decrease w.r.t. $k$. As a consequence, \eqref{eq:akj1_rem} and \eqref{ineq:akjdiff} result in 
$
 a_{j}^{(k+1)}- a_{j}^{(k)}>0,~ ( a_{j+1}^{(k+1)}- a_{j}^{(k+1)})- ( a_{j+1}^{(k)}- a_{j}^{(k)}) >0,
$
i.e., the properties (P2) and (P4) hold.

We now turn to prove the properties of $ c_{j}^{(k)}$ and $ d_{j}^{(k)}=  c_{j-1}^{(k)}- a_{j}^{(k)}$. For
$ c_{j}^{(k)}$ in \eqref{eq:aj}, we have 
\begin{equation}\label{eq:ckj_rem}
\begin{aligned}
  c_{j}^{(k)}  
   &=\frac{\tau_j^2}{\tau_{j+1}(\tau_{j}+\tau_{j+1})}\int_0^1(t_k-(t_{j-1}+s \tau_j))^{-\alpha}\ {\rm d}(s^2-s)\\
  &=\frac{\alpha\tau_{j}^3}{\tau_{j+1}(\tau_{j}+\tau_{j+1})}\int_0^{1} s(1-s)(t_k-t_j+s\tau_{j})^{-\alpha-1}\  \ds >0.
  \end{aligned}
\end{equation}
This is the property (P5).
Since $ a_{j}^{(k)}<0$, we have $ d_{j}^{(k)}= c_{j-1}^{(k)}- a_{j}^{(k)}>0$ for $j\geq 2$ and the property (P7) holds.
For any fixed $s$,
$(t_k-t_j+s\tau_{j})^{-\alpha-1}$
decreases w.r.t. $k$, implying that
$ c_{j}^{(k+1)}- c_{j}^{(k)}<0$,
i.e., the property (P6).
Combining this with property (P2), the property (P8) holds.

We now prove the property (P9).
Combining \eqref{ineq:akjdiff} and \eqref{eq:ckj_rem} gives 
\begin{align}\label{ineq:dkjdiff}
     d_{j+1}^{(k)}- d_{j}^{(k)}
    =&\frac{\alpha\tau_{j}^3}{\tau_{j+1}(\tau_{j}+\tau_{j+1})}\int_0^{1} s(1-s)(t_k-t_j+s\tau_{j})^{-\alpha-1}\  \ds\\
    &-\frac{\alpha\tau_{j-1}^3}{\tau_{j}(\tau_{j-1}+\tau_{j})}\int_0^{1} s(1-s)(t_k-t_{j-1}+s\tau_{j-1})^{-\alpha-1}\  \ds\nonumber\\
    &+\frac{\alpha\tau_j}{\tau_j+\tau_{j+1}}
    \int_0^{1} (\tau_j+\tau_{j+1}+s\tau_j)
    (1-s)(t_k-t_{j}+s\tau_j)^{-\alpha-1}\,\ds
   \nonumber \\
    &+\frac{\alpha\tau_{j+1}}{\tau_{j+1}+\tau_{j+2}}
     \int_0^{1}(\tau_{j+1}+\tau_{j+2}-s\tau_{j+1}) (1-s)(t_k-t_{j}-s\tau_{j+1})^{-\alpha-1}\,\ds.\nonumber
  \end{align}
Note that for any fixed $j$,
$(t_k-t_{j}+s\tau_j)^{-\alpha-1}\ge 0$
decreases w.r.t. $s$,
and 
$\int_0^1(1-3s)(1-s)\ds= 0$,
which imply
\begin{equation}\label{3.17}
\begin{aligned}
    &\int_0^{1} (\tau_j+\tau_{j+1}+s\tau_j)
    (1-s)(t_k-t_{j}+s\tau_j)^{-\alpha-1}\,\ds\\
    &\geq    \int_0^{1} (4\tau_j+3\tau_{j+1})s
    (1-s)(t_k-t_{j}+s\tau_j)^{-\alpha-1}\,\ds.
\end{aligned}
\end{equation}
Using \eqref{3.17} and the fact
$
    (t_k-t_j+s\tau_{j})^{-\alpha-1}>(t_k-t_{j-1}+s\tau_{j-1})^{-\alpha-1},
$
we can derive from \eqref{ineq:dkjdiff} that
\begin{equation}\label{eq:3.19}
\begin{aligned}
         d_{j+1}^{(k)}- d_{j}^{(k)}
        >&\alpha\bigg(\frac{\tau_{j}^3}{\tau_{j+1}(\tau_{j}+\tau_{j+1})}-\frac{\tau_{j-1}^3}{\tau_{j}(\tau_{j-1}+\tau_{j})}\\
        &\quad+\frac{(4\tau_j+3\tau_{j+1})\tau_j}{\tau_j+\tau_{j+1}}\bigg)\int_0^{1} s(1-s)(t_k-t_j+s\tau_{j})^{-\alpha-1}\  \ds.
\end{aligned}
\end{equation}
The property (P9) holds if the following condition is satisfied
\begin{equation*}
\small
\begin{aligned}
  &\frac{\tau_{j}^3}{\tau_{j+1}(\tau_{j}+\tau_{j+1})}-\frac{\tau_{j-1}^3}{\tau_{j}(\tau_{j-1}+\tau_{j})}+\frac{(4\tau_j+3\tau_{j+1})\tau_j}{\tau_j+\tau_{j+1}}\ge 0\\
  \iff &\frac{1}{\rho_{j+1}(1+\rho_{j+1})}-\frac{1}{\rho_{j}^2(1+\rho_{j})}+\frac{4+3\rho_{j+1}}{1+\rho_{j+1}}\ge0 \iff\frac{1}{\rho_{j+1}}\ge\frac{1}{\rho_{j}^2(1+\rho_{j})}-3.
\end{aligned}
\end{equation*}
We now prove the last property (P10).
The convexity of the function $t^{-\alpha-1}$ gives
\begin{equation*}
\begin{aligned}
    & (t_k-t_{j}+s\tau_{j})^{-\alpha-1}-(t_{k+1}-t_{j}+s\tau_{j})^{-\alpha-1}\\
    & >(t_k-t_{j-1}+s\tau_{j-1})^{-\alpha-1}-(t_{k+1}-t_{j-1}+s\tau_{j-1})^{-\alpha-1},
\end{aligned}
\end{equation*}
and for fixed $j$, it is  not difficult to see that
$
    (t_k-t_{j}+s\tau_{j})^{-\alpha-1}-(t_{k+1}-t_{j}+s\tau_{j})^{-\alpha-1}>0
$
decreases w.r.t. $s$.
Then we can get  the following result similar to \eqref{eq:3.19}:
\begin{equation*}
\small
\begin{aligned}
    &( d_{j+1}^{(k)}- d_{j}^{(k)})-( d_{j+1}^{(k+1)}- d_{j}^{(k+1)} )
    >\alpha\bigg(\frac{\tau_{j}^3}{\tau_{j+1}(\tau_{j}+\tau_{j+1})}-\frac{\tau_{j-1}^3}{\tau_{j}(\tau_{j-1}+\tau_{j})}
    \\
    &\quad+\frac{(4\tau_j+3\tau_{j+1})\tau_j}{\tau_j+\tau_{j+1}}\bigg)
    \int_0^{1} s(1-s)
     \bigg[(t_k-t_{j}+s\tau_{j})^{-\alpha-1}-(t_{k+1}-t_{j}+s\tau_{j})^{-\alpha-1}\bigg]\, \ds.
     \end{aligned}
\end{equation*}
Similar to the proof of (P9), 
$
    ( d_{j+1}^{(k)}- d_{j}^{(k)})-( d_{j+1}^{(k+1)}- d_{j}^{(k+1)} )>0,
$
as soon as the condition \eqref{condition:rho} is satisfied.
Therefore, (P10) is proved.
\end{proof}

\begin{theorem}\label{thm1}
Consider a nonuniform mesh $\{\tau_k\}_{k\geq 1}$ satisfying that 
\begin{equation}\label{thm1cond1}
 \begin{aligned}
 & \rho_*<\rho_2,\quad \rho_*<\rho_3<\rho^*, \quad 2+\frac{2}{1+\rho_3}+\frac{4\rho_2}{1+\rho_2}-\frac{\rho_2^3}{(1+\rho_2)^2}\ge 0,\\
 &\rho_{3} \le  \frac{\rho_2^2(1+\rho_2)}{1-3\rho_2^2(1+\rho_2)},
\end{aligned}
 \end{equation}
and for $k\geq 3$,
\begin{equation} \label{thm1cond2}
\left\{
\begin{aligned}
        &\rho_*<\rho_{k+1} \le  \frac{\rho_k^2(1+\rho_k)}{1-3\rho_k^2(1+\rho_k)},&&\text{if } \rho_* <\rho_k < \xi_1,\\
      &\rho_*<\rho_{k+1}<\rho^*,&&\text{if }  \xi_1 \le\rho_k \le\xi_2,\\
      &\rho_*<\rho_{k+1}\le \frac{-\rho_k^2+4\rho_k+2}{\rho_k^2-3\rho_k-1}, &&\text{if } \xi_2< \rho_k<\rho^*,
\end{aligned}
\right.
\end{equation}
where 
$\rho_*\approx 0.356341$  is the positive root of $\rho(1+\rho)=1-3\rho^2(1+\rho)$,
$
    \rho^*\approx4.155358
$
is defined in \eqref{eq:rho^*},
$
    \xi_1\approx0.459770
$
is the positive root of 
$
    \frac{\rho^2(1+\rho)}{1-3\rho^2(1+\rho)}=\rho^*
$
and
$
    \xi_2\approx3.532016
$
is the positive root of 
$
    \frac{-\rho^2+4\rho+2}{\rho^2-3\rho-1}=\rho^*.
$
Then  for any function $u$ defined on $[0,\infty)\times\Omega$ and $n\geq 2$,
\begin{equation*}\label{eq:PD}
    \mathcal B_n(u,u) = \sum_{k=1}^{n}\langle L_k^\alpha u, \delta_k u\rangle\geq \sum_{k=1}^{n} \frac{g_k(\alpha)}{2\Gamma(3-\alpha)} \|\delta_k u\|^2_{L^2(\Omega)}\geq 0,
\end{equation*}
where 
\begin{equation} \label{eq:gk}
\footnotesize
g_k(\alpha)= \\
\left\{
\begin{aligned}
& \frac{1}{\tau_1^\alpha} \hat{g}(\alpha), &&k=1,2,\\
       & \frac{\alpha}{\tau_k^{\alpha}}\left(\frac{(1+\alpha)(2-\alpha)}{\alpha}
    +\frac{\rho_{k+1}-1+\alpha}{(1+\rho_{k+1})^{\alpha}}
    -\frac{2\rho_{k+1}^{2-\alpha}}{1+\rho_{k+1}}-\frac{\rho_k(\rho_k-2)}{1+\rho_k}\right), &&3\le k\le n-1,\\
    &\frac{\alpha}{\tau_n^{\alpha}}\left(\frac{(1+\alpha)(2-\alpha)}{\alpha} -\frac{\rho_n(\rho_n-2)}{1+\rho_n}\right),&&k=n\neq 2,
\end{aligned}
\right.
\end{equation}
are positive for all $\alpha \in (0,1)$ and $\hat g(\alpha)$ is defined in \eqref{ineq:ghat} depending only on $\alpha,~\rho_2,~\rho_3$.
\end{theorem}
 \vspace{-0.1in}
\begin{figure}[!ht]
    \centering
    \includegraphics[trim={1.5in 0.3in 1.5in 0.8in},clip,width=0.64\textwidth]{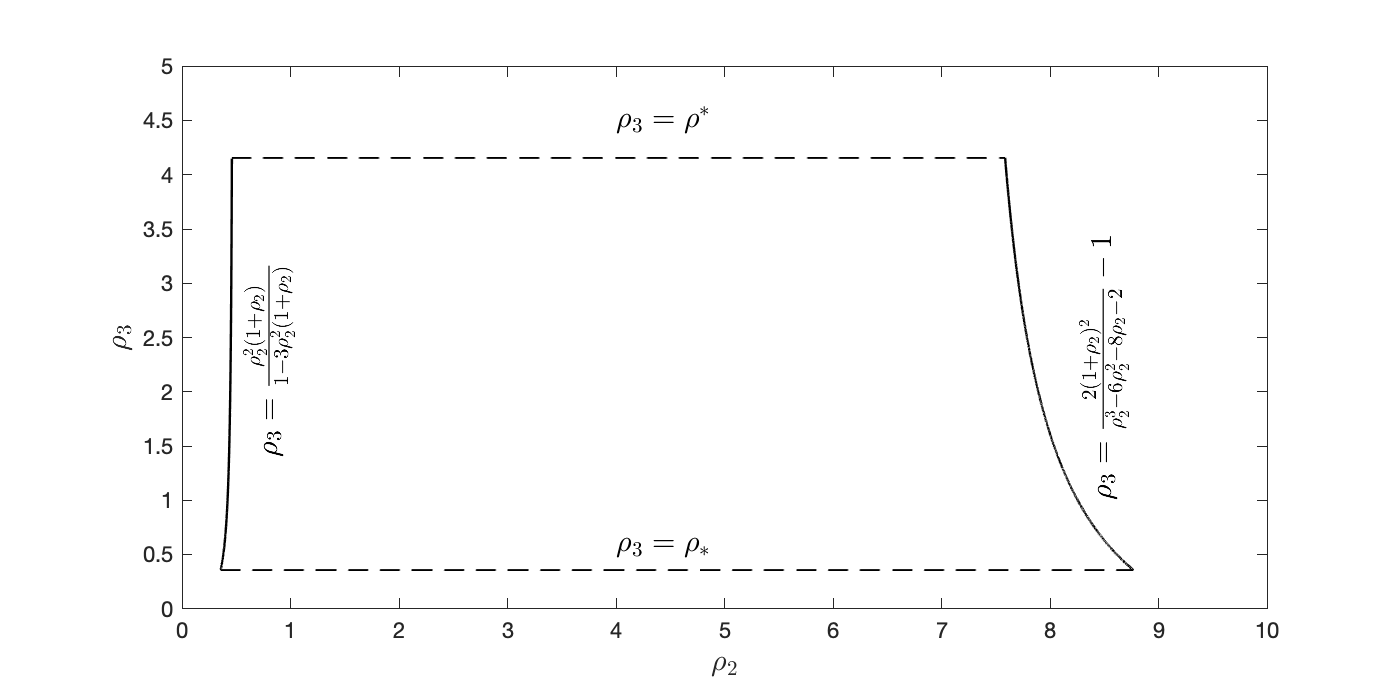}
    \includegraphics[trim={0.5in 0 1.in 0.6in},clip,width=0.35\textwidth]{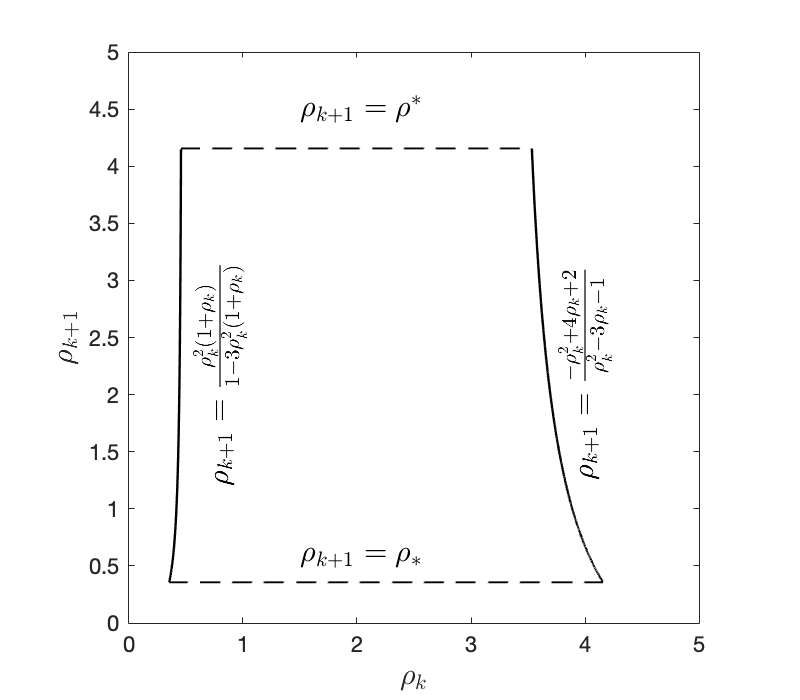}
    \vspace{-0.2in}
    \caption{Feasible regions from restriction \eqref{thm1cond1} for $(\rho_2,\rho_3)$ (left) and restriction \eqref{thm1cond2} for $(\rho_k,\rho_{k+1})$ with $k\geq 3$ (right).}
    \label{fig:rhok}
\end{figure}
\begin{proof}
The graphical illustration of restrictions \eqref{thm1cond1} and \eqref{thm1cond2} are provided in Figure \ref{fig:rhok}.
According to \eqref{eq:Lk1}, we can rewrite $\mathcal B_n(u,u)$ in the following matrix form
\begin{equation*}
    \mathcal B_n(u,u) = \sum_{k=1}^{n}\langle L_k^\alpha u, \delta_k u\rangle=\frac{1}{\Gamma(1-\alpha)}\int_{\Omega}\Psi \mathbf M \Psi^{\mathrm{T}}\dx,
\end{equation*}
where 
$
    \Psi=[\delta_{1} u,\delta_2 u,\cdots,\delta_n u]
$
and
\begin{equation}\label{matM}
\footnotesize\mathbf M=
\begin{pmatrix}
(1-\alpha)^{-1}\tau_1^{-\alpha} & \\
-a_1^{(2)}-a_2^{(2)}& c^{(2)}_1+c^{(2)}_2 \\
-a_1^{(3)}& d^{(3)}_2-a_3^{(3)}& c^{(3)}_2+c^{(3)}_3 \\
-a_1^{(4)} & d^{(4)}_2 & d_3^{(4)}-a_4^{(4)} & c^{(4)}_3+c^{(4)}_4\\
\vdots& \vdots & \ddots &\ddots & \ddots\\
-a_1^{(n)}& d^{(n)}_2&\cdots& d^{(n)}_{n-2}& d^{(n)}_{n-1}-a^{(n)}_{n}&c^{(n)}_{n-1}+c^{(n)}_n 
\end{pmatrix}.
\end{equation}
We split $\mathbf M$ as $\mathbf M = \mathbf A+\mathbf B$, where

\begin{equation*}
  \footnotesize  \mathbf A=
\begin{pmatrix}
\beta_1 & \\
-a_1^{(2)}& \beta_2  \\
-a_1^{(3)} & d_2^{(3)}& \beta_3 \\
\vdots& \vdots &\ddots & \ddots\\
-a_1^{(n)}& d_2^{(n)}&\cdots& d^{(n)}_{n-1}& \beta_n
\end{pmatrix},
\end{equation*}
and
\begin{equation*}
\footnotesize    \mathbf B=
\begin{pmatrix}
(1-\alpha)^{-1}\tau_1^{-\alpha}-\beta_1& \\
-a_2^{(2)} &  c_1^{(2)}+c_2^{(2)}-\beta_2\\
 & -a_3^{(3)}& c_2^{(3)}+c_3^{(3)}-\beta_3\\
&  &  \ddots & \ddots\\
& & &  -a_n^{(n)}& c^{(n)}_{n-1}+c^{(n)}_n-\beta_n
\end{pmatrix},
\end{equation*}
with
\begin{equation}\label{betak}
\begin{aligned}
    &2\beta_1= -a_1^{(2)},\quad 2\beta_2 -d_2^{(3)}=a_1^{(3)}-a_1^{(2)},\\
    &2\beta_k -d_k^{(k+1)}=d_{k-1}^{(k)}-d_{k-1}^{(k+1)},\quad 3\leq k\leq n-1,\\
    &2\beta_n=d_{n-1}^{(n)},\quad n\geq 3.
\end{aligned}
\end{equation}

Consider the following symmetric matrix 
$
    \mathbf S = \mathbf A+\mathbf A^{\mathrm T}+\varepsilon \mathbf e_n^{\rm T}\mathbf e_n,
$
with small constant $\varepsilon>0$ and $\mathbf e_n = (0,\cdots,0,1)\in \mathbb R^{1\times n}$.
According to Lemma \ref{lemmapro}, if the condition \eqref{condition:rho} holds,  $\mathbf S$ satisfies the following three properties:
\begin{itemize}\label{item:1}
\item[{\rm (1)}] $\forall\; 1\leq j < i \leq n$, $\left[ \mathbf S \right]_{i-1,j}\geq \left[ \mathbf S \right]_{i, j}$;
\item[{\rm (2)}] $\forall\; 1 < j \leq i \leq n$, $\left[ \mathbf S \right]_{i, j-1}< \left[ \mathbf S \right]_{i, j}$;
\item[{\rm (3)}] $\forall \;1< j < i \leq n$, $\left[ \mathbf S \right]_{i-1, j-1} - \left[ \mathbf S \right]_{i, j-1}\leq \left[ \mathbf S \right]_{i-1, j} - \left[ \mathbf S \right]_{i, j}$.
\end{itemize}
From \cite[Lemma 2.1]{CSIAM-AM-1-478}, $\mathbf S$ is positive definite. 
Let $\varepsilon \rightarrow 0$. 
We can claim that
$\mathbf A+\mathbf A^{\mathrm T}$
is positive semidefinite. 

We now consider the following splitting of $\mathbf B+\mathbf B^{\mathrm{T}}$:
{
\begin{equation*}
    \mathbf B + {\mathbf B}^{\rm T}=
    \begin{pmatrix}
\mathbf C& \mathbf 0\\
\mathbf 0 &  \mathbf 0\\
\end{pmatrix}_{ n\times n}+
\begin{pmatrix}
0& \mathbf 0\\
\mathbf 0 &  \mathbf D\\
\end{pmatrix}_{n\times n},
\end{equation*}}
where
\begin{equation*}
\small
\mathbf C=\begin{pmatrix}
2(1-\alpha)^{-1}\tau_1^{-\alpha}-2\beta_1& -a_2^{(2)}\\
-a_2^{(2)} &  2c_1^{(2)}+2c_2^{(2)}-2\beta_2-a_3^{(3)}\\
\end{pmatrix}_{2\times2},
\end{equation*}
\begin{equation*}
\small
\begin{aligned}
    &\mathbf D=
 \begin{pmatrix}
 a^{(3)}_3&-a_3^{(3)}\\
  -a_3^{(3)}& 2c_2^{(3)}+2c_3^{(3)}-2\beta_3&-a^{(4)}_4\\
  &  \ddots & \ddots&\ddots\\
 & &  -a_n^{(n)}& 2c^{(n)}_{n-1}+2c^{(n)}_n-2\beta_n
\end{pmatrix}_{(n-1)\times (n-1)}.
\end{aligned}
\end{equation*}
The positive semidefiniteness of $\mathbf B+\mathbf B^{\mathrm{T}}$ can be ensured if $\mathbf C$ and $\mathbf D$ are both positive semidefinite. 

We first discuss about the positive semidefiniteness of $\mathbf C$ of size $2\times2$.
Note that from \eqref{eq:akj}, we have the following explicit expression of $a_1^{(2)}$:
\begin{equation}\label{a1122}
\begin{aligned}
   a_1^{(2)}
   & = \frac{\tau_{2}^{2-\alpha}}{(1-\alpha)\tau_1(\tau_1+\tau_{2})}-\frac{(2\tau_{1}+\tau_{2})(\tau_{1}+\tau_{2})^{1-\alpha}}{(1-\alpha)\tau_1(\tau_1+\tau_{2})} +\frac{2\left[(\tau_{1}+\tau_{2})^{2-\alpha}-\tau_2^{2-\alpha}\right]}{(2-\alpha)(1-\alpha)\tau_1(\tau_1+\tau_{2})}\\
   &=\frac{\alpha(\tau_{1}+\tau_{2})^{2-\alpha}-\alpha\tau_2^{2-\alpha}}{(2-\alpha)(1-\alpha)\tau_1(\tau_1+\tau_{2})}-(1-\alpha)^{-1}(\tau_1+\tau_2)^{-\alpha},
  \end{aligned}
\end{equation}
and from \eqref{eq:akj_rem}, we have another formula of $a_1^{(2)}$:
\begin{equation}\label{a11221}
   a_1^{(2)}= -\tau_2^{-\alpha}+\frac{\alpha\tau_1}{\tau_1+\tau_{2}}
    \int_0^{1} (\tau_1+\tau_{2}+s\tau_1)
    (1-s)(\tau_2+s\tau_1)^{-\alpha-1}\,\ds.
\end{equation}
According to the definition $2\beta_1=-a_1^{(2)}$ in \eqref{betak}, if $\tau_1\ge\tau_2$, then by \eqref{a1122},
{ 
\begin{equation}\label{tau21}
\begin{aligned}
    [\mathbf C]_{11}& 
    >2(1-\alpha)^{-1}\tau_1^{-\alpha}-(1-\alpha)^{-1}(\tau_1
    +\tau_2)^{-\alpha} >(1-\alpha)^{-1}\tau_1^{-\alpha}>0,
\end{aligned}
\end{equation}}
while if $\tau_1\le\tau_2$, then by \eqref{a11221}
\begin{equation}\label{tau12}
\begin{aligned}
    [\mathbf C]_{11} 
     =2(1-\alpha)^{-1}\tau_1^{-\alpha}+a_1^{(2)}
    >\frac{1+\alpha}{(1-\alpha)\tau_1^{\alpha}}>0.
\end{aligned}
\end{equation}
From $2\beta_2= d_2^{(3)}+a_1^{(3)}-a_1^{(2)}$ in \eqref{betak}, $d_2^{(3)}=c_1^{(3)}-a_2^{(3)}$ and the properties (P5)--(P6) on $ c_{j}^{(k)}$, we have 
 \begin{equation}\label{c123}
\begin{aligned}
     [\mathbf C]_{22} 
    =&c_1^{(2)}+2c_2^{(2)}+(a_1^{(2)}+a_2^{(3)}-a_1^{(3)})-a_3^{(3)}+(c^{(2)}_1-c_1^{(3)})\\
    >&2c_2^{(2)}+(a_1^{(2)}+a_2^{(3)}-a_1^{(3)})-a_3^{(3)}.
\end{aligned}
\end{equation}
Note that \eqref{eq:akj_rem} and \eqref{ineq:akjdiff} give
\begin{equation}\label{a123}
\begin{aligned}
    a_1^{(2)}+a_2^{(3)}-a_1^{(3)}=&-\tau_2^{-\alpha}+\frac{\alpha\tau_1}{\tau_1+\tau_2}
    \int_0^{1} (\tau_1+\tau_{2}+s\tau_1)
    (1-s)(t_2-t_{1}+s\tau_1)^{-\alpha-1}\,\ds\\
    &-\frac{\alpha\tau_1}{\tau_1+\tau_{2}}
    \int_0^{1} (\tau_1+\tau_{2}+s\tau_1)
    (1-s)(t_3-t_{1}+s\tau_1)^{-\alpha-1}\,\ds
    \\
    &-\frac{\alpha\tau_{2}}{\tau_{2}+\tau_{3}}
     \int_0^{1}(\tau_{2}+\tau_{3}-s\tau_{2}) (1-s)(t_3-t_{1}-s\tau_2)^{-\alpha-1}\,\ds\\
     >&-\tau_2^{-\alpha}-\frac{\alpha\tau_{2}}{\tau_{2}+\tau_{3}}
     \int_0^{1}(1-s)(\tau_{2}+\tau_{3}-s\tau_{2}) ^{-\alpha}\,\ds\\
    =&-\tau_2^{-\alpha}-\frac{\alpha}{(2-\alpha)(1-\alpha)\tau_{2}^\alpha}\left(-\frac{\rho_{3}-1+\alpha}{(1+\rho_{3})^{\alpha}}+\frac{\rho_{3}^{2-\alpha}}{1+\rho_{3}}\right).
\end{aligned}
\end{equation}
Substituting \eqref{eq:a_k1} and \eqref{a123} into \eqref{c123} yields
 \begin{equation}\label{eq:ck2}
\begin{aligned}
     [\mathbf C]_{22} 
     >\frac{\alpha}{(2-\alpha)(1-\alpha)\tau_{2}^{\alpha}}\bigg(
    \frac{(1+\alpha)(2-\alpha)}{\alpha}+\frac{\rho_{3}-1+\alpha}{(1+\rho_{3})^{\alpha}}-\frac{2\rho_{3}^{2-\alpha}}{1+\rho_{3}}+\frac{2\rho_2}{1+\rho_2}\bigg).
\end{aligned}
\end{equation}
Since $[\mathbf C]_{11}>0$, $\mathbf C$ is positive definite as soon as 
$
   [\mathbf C]_{11}  [\mathbf C]_{22} -[\mathbf C]_{12} [\mathbf C]_{21} > 0.
$

When $\tau_1\ge\tau_2$, i.e. $\rho_2\le1$, from \eqref{eq:a_k1}, \eqref{tau21} and \eqref{eq:ck2}, we have
\begin{align*}
     [\mathbf C]_{11}  [\mathbf C]_{22} -[\mathbf C]_{12} [\mathbf C]_{21}
     > \frac{\alpha}{(1-\alpha)^2(2-\alpha)(\tau_1\tau_2)^\alpha} h_1(\alpha),
\end{align*}
where
\begin{align*}
    h_1(\alpha)&=
   \frac{(1+\alpha)(2-\alpha)}{\alpha} +\frac{\rho_{3}-1+\alpha}{(1+\rho_{3})^{\alpha}}
   -\frac{2\rho_{3}^{2-\alpha}}{1+\rho_{3}}+\frac{2\rho_2}{1+\rho_2}-\frac{\alpha}{(2-\alpha)\rho_2^{\alpha}}\left(\frac{\rho_2^2}{1+\rho_2}\right)^2.
  \end{align*}
Now we show that $h_1(\alpha)$ decreases w.r.t. $\alpha$ and $h_1(1)\geq 0$ under some constraints on  $\rho_2$ and $\rho_3$.
It is easy to check that 
$
    -\frac{\alpha}{(2-\alpha)\rho_2^{\alpha}}\left(\frac{\rho_2^2}{1+\rho_2}\right)^2
$
decreases w.r.t. $\alpha$ when $\rho_2\leq 1$. 
Let 
\begin{equation}\label{def:q}
    q(\alpha)=
    \frac{(1+\alpha)(2-\alpha)}{\alpha} +\frac{\rho_{3}-1+\alpha}{(1+\rho_{3})^{\alpha}}
   -\frac{2\rho_{3}^{2-\alpha}}{1+\rho_{3}}+\frac{2\rho_2}{1+\rho_2}.
\end{equation}
A direct calculation gives 
\begin{equation}\label{qprime}
    q'(\alpha) =-2/\alpha^2-1+(1+\rho_{3})^{-\alpha}-\frac{(\rho_{3}-1+\alpha)\ln (1+\rho_{3})}{(1+\rho_{3})^{\alpha}}+\frac{2\rho_{3}^{2-\alpha}\ln (\rho_{3})}{1+\rho_{3}}.
\end{equation}
To show $q'(\alpha) \leq 0$, we consider the following several cases. 
In the case of $0<\rho_3\le1$, we have
\begin{align*}
    q'(\alpha)\le-2/\alpha^2-1+(1+\rho_{3})^{-\alpha}(1-(\rho_3-1+\alpha)\ln(1+\rho_3))
     \le -3+(1+\ln2)\le0.
     \end{align*}
In the case of $1<\rho_3\leq 4.5$, we have
\begin{align*}
    q'(\alpha) \leq -2/\alpha^2-1+2^{-\alpha}+\frac{2\times 4.5^{2-\alpha}\ln (4.5)}{1+4.5}\leq 0.
\end{align*}
So $q(\alpha)$ decreases w.r.t. $\alpha$ for $0<\rho_3\leq 4.5$. As a consequence, $h_1(\alpha)$ decreases w.r.t. $\alpha$ for $0<\rho_3\leq 4.5$. 
Since
    \begin{align*}
             h_1(1)
   =2-\frac{\rho_{3}}{1+\rho_{3}}+\frac{2\rho_2}{1+\rho_2}-\frac{\rho_2^3}{(1+\rho_2)^2}
   =1+\frac{1}{1+\rho_{3}}+\frac{2\rho_2+2\rho_2^2-\rho_2^3}{(1+\rho_2)^2}
   > 0,
    \end{align*}
we know that $\mathbf C$ is positive definite for $\alpha \in (0,1)$ when $\rho_2\leq 1$ and $\rho_3\leq 4.5$.
 
When $\tau_1\le\tau_2$, i.e. $\rho_2\ge1$, from \eqref{eq:a_k1}, \eqref{tau12} and \eqref{eq:ck2}, we have
\begin{align*}
     [\mathbf C]_{11}  [\mathbf C]_{22} -[\mathbf C]_{12} [\mathbf C]_{21}
     > \frac{\alpha^2}{(1-\alpha)^2(2-\alpha)^2\tau_2^{2\alpha}} h_2(\alpha),
\end{align*}
where
\begin{align*}
h_2(\alpha)=&\frac{(1+\alpha)(2-\alpha)\rho_2^\alpha}{\alpha}q(\alpha)
   -\left(\frac{\rho_2^2}{1+\rho_2}\right)^2
 \end{align*}
with $q(\alpha)$ defined in \eqref{def:q}. We want to impose some constraints on $\rho_2$ and $\rho_3$ s.t. $h_2(\alpha)\geq 0$ for $\alpha\in (0,1)$.
 First we have to impose $h_2(1)\geq 0$, i.e.,
 \begin{equation}\label{condc11}
     \rho_2^{-1} h_2(1) = 2+\frac{2}{1+\rho_3}+\frac{4\rho_2}{1+\rho_2}-\frac{\rho_2^3}{(1+\rho_2)^2}\ge 0,
 \end{equation}
 which is equivalent to
 \begin{align}\label{eq:psirho}
     \rho_2^3-\left(6+\frac{2}{1+\rho_3}\right)\rho_2^2-\left(8+\frac{4}{1+\rho_3}\right)\rho_2-\left(2+\frac{2}{1+\rho_3}\right)\leq 0.
 \end{align}
 Solving this cubic inequality yields
 \begin{equation}\label{eq:rhobar}
 \begin{aligned}
     0<\rho_2 \leq \psi(\rho_3),
\end{aligned}
 \end{equation}
 where $\psi(\rho_3)$ is the unique positive root of the left-hand side of \eqref{eq:psirho}.
 Next, we show that under the the constraint \eqref{condc11}, $h_2'(\alpha)\leq 0$ holds, so that $h_2(\alpha)\geq 0$ for $\alpha\in (0,1)$.
 Note that \eqref{condc11} indicates
\begin{equation}\label{eq:3.70}
    \frac{\rho_2^3}{(1+\rho_2)^2}-\frac{4\rho_2}{1+\rho_2}\leq 2+\frac{2}{1+\rho_3}< 4 \quad \Rightarrow \quad\rho_2< 9.331852.
\end{equation}
A direct computation gives
 $ h_2'(\alpha)=\rho_2^{\alpha} q(\alpha) p(\alpha)$,
 where $q(\alpha)$ is defined in \eqref{def:q} and
 \begin{equation*}
     p(\alpha) = -2/ \alpha^2-1+\ln {\rho_2}\frac{(1+\alpha)(2-\alpha)}{\alpha}+ \frac{(1+\alpha)(2-\alpha) q'(\alpha)}{\alpha q(\alpha)}.
 \end{equation*}
Recall that $ q'(\alpha)\leq 0$ for $\rho_2> 0$, $0<\rho_3\leq 4.5$, implying that $q(\alpha)\ge q(1) > 0$. 
We now prove that $p(\alpha)\leq 0$ for any $\alpha\in(0,1)$,  $0.3<\rho_3\leq 4.5$ and  $0<\rho_2 \le\psi(\rho_3)$.
The following three cases are discussed. \\
{\bf Case 1:} $0< \alpha \le 0.43$. We have the following estimate
  \begin{align*}
      p(\alpha) \leq &-2/ \alpha^2-1+\ln 9.4\frac{(1+\alpha)(2-\alpha)}{\alpha} \le0,\quad 0<\rho_2\leq 9.4,~0<\rho_3\leq 4.5,
  \end{align*}
where we use the inequality \eqref{eq:3.70}. \\
{\bf Case 2:} $0.43< \alpha\leq 0.7$.
In this case, for any $1\leq\rho_2<9.4$,
  \begin{equation*}
  \begin{aligned}
     & \frac{\rm d}{{\rm d}\alpha}\left(-2/ \alpha^2-1+\ln {\rho_2}\frac{(1+\alpha)(2-\alpha)}{\alpha}\right)
= 4/ \alpha^3+\ln {\rho_2}(-2/ \alpha^2-1) \\
&>4/ \alpha^3+\ln {9.4}(-2/ \alpha^2-1)
      \geq  4/\alpha^3-4.5/\alpha^2-2.25\geq (4/0.7-4.5)/0.7^2-2.25\geq 0.
      \end{aligned}
  \end{equation*}
For any interval $(b,a]\subset (0.43,0.7]$ and $\alpha\in(b,a]$, we have the following upper bound
\begin{equation}\label{palpha1}
   \begin{aligned}  
   p(\alpha)\le&\left(-2/ a^2-1+\ln {\rho_2}\frac{(1+a)(2-a)}{a}\right)
   +\frac{(1+a)(2-a)}{a q(b)}\bigg(-2/a^2-1\\
   &+(1+\rho_{3})^{-b}
   -\frac{\rho_{3}\ln (1+\rho_{3})}{(1+\rho_{3})^{a}}+\frac{(1-b)\ln (1+\rho_{3})}{(1+\rho_{3})^{b}}+\frac{2\rho_{3}^{2-b}}{1+\rho_{3}}\ln (\rho_{3})\bigg)\\
   =:& \varphi_1(\rho_2,\rho_3)\le\varphi_1(\psi(\rho_3),\rho_3) .
   \end{aligned}
   \end{equation}
Here, $\psi(\rho_3)$ is defined in \eqref{eq:rhobar} and we use the fact that $\varphi_1(\rho_2,\rho_3)$ increases w.r.t. $\rho_2$.
We separate $(0.43,0.7]$ into $(0.43,0.6]$ and $(0.6,0.7]$, and plot the upper bounds according to \eqref{palpha1} on these two small intervals respectively (see the left-hand side of Figure \ref{fig:p}). Both upper bounds are smaller than $0$. So $p(\alpha)\leq 0$ for $0.43<\alpha\leq  0.7,~0.3<\rho_3\leq 4.5,~0<\rho_2\leq \psi(\rho_3)$.\\
{\bf Case 3:} $0.7< \alpha<1$.
For any interval $(b,a]\subset (0.7,1]$ and $\alpha\in(b,a]$,
\begin{equation}\label{palpha2}
   \begin{aligned}  
   p(\alpha)\le&\left(-2/ a^2-1+\ln {\rho_2}\frac{(1+b)(2-b)}{b}\right)
   +\frac{(1+a)(2-a)}{a q(b)}\bigg(-2/a^2-1\\
   &+(1+\rho_{3})^{-b}
   -\frac{\rho_{3}\ln (1+\rho_{3})}{(1+\rho_{3})^{a}}+\frac{(1-b)\ln (1+\rho_{3})}{(1+\rho_{3})^{b}}+\frac{2\rho_{3}^{2-b}}{1+\rho_{3}}\ln (\rho_{3})\bigg)\\
   =:&  \varphi_2(\rho_2,\rho_3)\le\varphi_2(\psi(\rho_3),\rho_3).
   \end{aligned}
   \end{equation}
   Here, $\psi(\rho_3)$ is defined in \eqref{eq:rhobar} and we use the fact that $\varphi_2(\rho_2,\rho_3)$ increases w.r.t. $\rho_2$.
 We separate $(0.7,1]$ into small intervals and plot the upper bounds according to \eqref{palpha2} on all these small intervals respectively (see the 
right-hand side of Figure \ref{fig:p}). All these upper bounds are smaller than $0$. So $p(\alpha)\leq 0$ for $0.7<\alpha\leq  1,~0.3<\rho_3\leq 4.5,~0<\rho_2\leq \psi(\rho_3)$.
 \begin{figure}
    \centering
    \includegraphics[trim={0 0in 0 .5in},clip,width=0.95\textwidth]{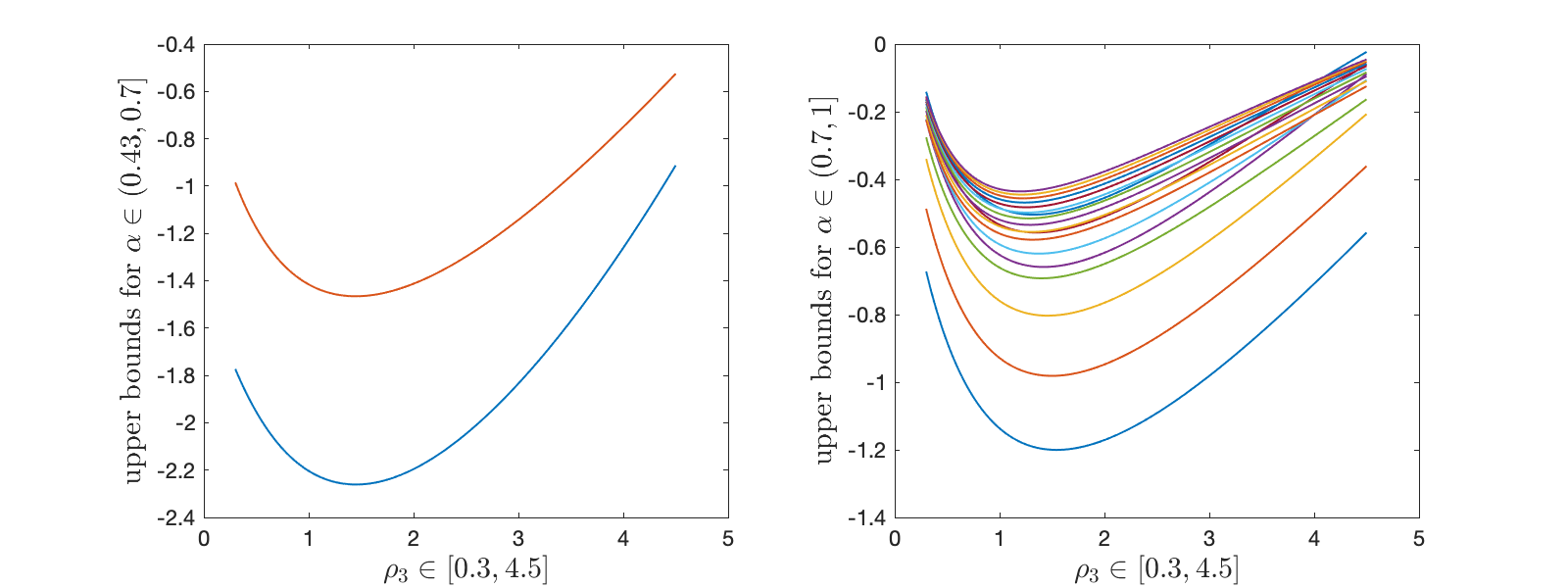}
    \caption{Left: Upper bounds of $p(\alpha)$ for $(0.43,0.6]$, $(0.6,0.7]$ from \eqref{palpha1}; Right: Upper bounds of $p(\alpha)$ for  $(0.7,0.73]$, $(0.73,0.76]$, $(0.76,0.79]$, $(0.79,0.82]$,  $(0.82,0.84]$,  $(0.84,0.86]$, $(0.86,0.88]$, $(0.88,0.9]$, $(0.9,0.91]$,  $(0.91,0.92]$, $(0.92,0.93]$, $(0.93,0.94]$,  $(0.94,0.95]$, $(0.95,0.96]$,   $(0.96,0.97]$, $(0.97,0.98]$, $(0.98,0.99]$, $(0.99,1]$ from \eqref{palpha2}.}
    \label{fig:p}
\end{figure}
Combining {\bf Case 1--3},  we derive that $\mathbf C$ is positive definite for $\alpha \in (0,1)$ when $\rho_2 \geq 1$, $0.3\leq \rho_3\leq 4.5$ and \eqref{condc11} is satisfied.

Combining all above discussions for $0<\rho_2\le 1$ and $\rho_2\ge1$, we claim that if $0.3\leq \rho_3\leq 4.5$ and \eqref{condc11} is satisfied, then $\mathbf C$ is positive definite. Moreover, the eigenvalues of  $\mathbf C$ are
\begin{equation}\label{ineq:ghat}
\begin{aligned}
     \lambda_{1,2}&=\frac{[\mathbf C]_{11}+[\mathbf C]_{22}\pm \sqrt{([\mathbf C]_{11}+[\mathbf C]_{22})^2-4([\mathbf C]_{11}[\mathbf C]_{22}-[\mathbf C]_{12}[\mathbf C]_{21})}}{2}\\
     &\ge\frac{[\mathbf C]_{11}+[\mathbf C]_{22}- \sqrt{([\mathbf C]_{11}-[\mathbf C]_{22})^2+4 [\mathbf C]_{12}[\mathbf C]_{21}}}{2}\\
     &=: (2-\alpha)^{-1}(1-\alpha)^{-1}\tau_1^{-\alpha}\hat g(\alpha),
\end{aligned}
\end{equation}
where $\hat g(\alpha)>0$ depends only on $\alpha, ~\rho_2,~\rho_3$.

We have studied the positive definiteness of $\mathbf C$ and now turn to analyze the positive semidefiniteness of $\mathbf D$.
We aim to show that $\mathbf D$ is diagonally dominant under some constraints on $\rho_k$, so that the positive semidefiniteness can be  guaranteed. 
  
For $3\le k \le n-1$,  we show $2 c_{k-1}^{(k)}+2c_{k}^{(k)}-2\beta_k-a_k^{(k)}-a_{k+1}^{(k+1)}
  \ge 0$ under some constraints on $\rho_k$ and $\rho_{k+1}$. 
From the definition $2\beta_k=d_k^{(k+1)}+d_{k-1}^{(k)}-d_{k-1}^{(k+1)}$ in \eqref{betak} and $ d_{j}^{(k)}=c_{j-1}^{(k)}- a_{j}^{(k)}$, we have
\begin{equation} \label{eq:kn}
\footnotesize
\begin{aligned}
     &2 c_{k-1}^{(k)}+2c_{k}^{(k)}-2\beta_k-a_k^{(k)}-a_{k+1}^{(k+1)} \\
    =& 2 c_{k-1}^{(k)}+2c_{k}^{(k)}-d_k^{(k+1)}-d_{k-1}^{(k)}+d_{k-1}^{(k+1)}-a_k^{(k)}-a_{k+1}^{(k+1)}\\
    =&  c_{k-1}^{(k)}+2c_{k}^{(k)}+[(c_{k-1}^{(k)}-c_{k-1}^{(k+1)})-(c_{k-2}^{(k)}-c^{(k+1)}_{k-2} )+(a^{(k+1)}_{k}-a^{(k+1)}_{k-1})+a^{(k)}_{k-1}]-a_k^{(k)}-a_{k+1}^{(k+1)}.
\end{aligned}
\end{equation}
From  \eqref{eq:ckj_rem},  \eqref{ineq:akjdiff} and \eqref{eq:akj_rem}, we have
{\footnotesize
\begin{align}\label{eq:3.80-1}
    &(c_{k-1}^{(k)}-c_{k-1}^{(k+1)})-(c_{k-2}^{(k)}-c^{(k+1)}_{k-2} )+(a^{(k+1)}_{k}-a^{(k+1)}_{k-1})+a^{(k)}_{k-1}\\
    =&\frac{\alpha\tau_{k-1}^3}{\tau_{k}(\tau_{k-1}+\tau_{k})}\int_0^{1} s(1-s)\bigg[(t_k-t_{k-1}+s\tau_{k-1})^{-\alpha-1}
     -(t_{k+1}-t_{k-1}+s\tau_{k-1})^{-\alpha-1}\bigg]\  \ds\nonumber\\
   &-\frac{\alpha\tau_{k-2}^3}{\tau_{k-1}(\tau_{k-2}+\tau_{k-1})}\int_0^{1} s(1-s)\bigg[(t_k-t_{k-2}+s\tau_{k-2})^{-\alpha-1}-(t_{k+1}-t_{k-2}+s\tau_{k-2})^{-\alpha-1}\bigg]\  \ds\nonumber\\
   &+\frac{\alpha\tau_{k-1}}{\tau_{k-1}+\tau_{k}}
    \int_0^{1} (\tau_{k-1}+\tau_{k}+s\tau_{k-1})
    (1-s)\bigg[(t_k-t_{k-1}+s\tau_{k-1})^{-\alpha-1}\nonumber\\
    &\quad-(t_{k+1}-t_{k-1}+s\tau_{k-1})^{-\alpha-1}\bigg]\,\ds-\tau_k^{-\alpha}-\frac{\alpha\tau_{k}}{\tau_{k}+\tau_{k+1}}
     \int_0^{1} (1-s)(\tau_{k}+\tau_{k+1}-s\tau_{k})^{-\alpha}\,\ds\nonumber\\
   >&-\tau_k^{-\alpha}-\frac{\alpha\tau_{k}}{\tau_{k}+\tau_{k+1}}
     \int_0^{1} (1-s)(\tau_{k}+\tau_{k+1}-s\tau_{k})^{-\alpha}\,\ds\nonumber\\
    =&-\tau_k^{-\alpha}-\frac{\alpha}{(2-\alpha)(1-\alpha)\tau_{k}^{\alpha}}\left(-\frac{\rho_{k+1}-1+\alpha}{(1+\rho_{k+1})^{\alpha}}+\frac{\rho_{k+1}^{2-\alpha}}{1+\rho_{k+1}}\right),\nonumber
    \end{align}}
as soon as \eqref{condition:rho} is satisfied for $j=k-1$.
Here, the inequality in \eqref{eq:3.80-1} is obtained similar to the proof of the property (P10) in Lemma \ref{lemmapro}.
Combining this with  \eqref{eq:a_k1} and \eqref{eq:kn} yields
\begin{align*}
    2 c_{k-1}^{(k)}+2c_{k}^{(k)}-2\beta_k-a_k^{(k)}-a_{k+1}^{(k+1)}
    \ge  c_{k-1}^{(k)}+\frac{\alpha h_3(\alpha)}{(2-\alpha)(1-\alpha)\tau_{k}^{\alpha}}
    ,
\end{align*}
where
\begin{align*}
    h_3(\alpha)= \frac{(1+\alpha)(2-\alpha)}{\alpha} 
    +\frac{\rho_{k+1}-1+\alpha}{(1+\rho_{k+1})^{\alpha}}
    -\frac{2\rho_{k+1}^{2-\alpha}}{1+\rho_{k+1}}-\frac{\rho_k(\rho_k-2)}{1+\rho_k}.
\end{align*}
A direct calculation gives
\begin{equation*}
    h_3'(\alpha) =-2/\alpha^2-1+(1+\rho_{k+1})^{-\alpha}-\frac{(\rho_{k+1}-1+\alpha)\ln (1+\rho_{k+1})}{(1+\rho_{k+1})^{\alpha}}+\frac{2\rho_{k+1}^{2-\alpha}}{1+\rho_{k+1}}\ln (\rho_{k+1}),
\end{equation*}
which is similar to $q'(\alpha)$ in \eqref{qprime} (just replacing $\rho_3$ by $\rho_{k+1}$).
Therefore, we have $ h_3'(\alpha)\le0$ and then $h_3(\alpha)\geq h_3(1)$ when $0<\rho_{k+1}<4.5$.
To ensure 
$
    2 c_{k-1}^{(k)}+2c_{k}^{(k)}-2\beta_k-a_k^{(k)}-a_{k+1}^{(k+1)}
  \ge 0,
$
it is sufficient to impose for $3\le k\le n-1$
\begin{equation}\label{rhokk1}
\begin{aligned}
    \frac{1}{\rho_{k}}\ge\frac{1}{\rho_{k-1}^2(1+\rho_{k-1})}-3,\quad 0<\rho_{k+1}<4.5,\quad h_3(1)=\frac{2+\rho_{k+1}}{1+\rho_{k+1}}-\frac{\rho_k(\rho_k-2)}{1+\rho_k}\ge0.
    \end{aligned}
\end{equation}

Now we show $2c^{(n)}_{n-1}+2c^{(n)}_n-2\beta_n-a_n^{(n)}\ge 0$ under some constraints on $\rho_n$.
From \eqref{betak}, \eqref{eq:a_k1}, \eqref{eq:akj_rem} and  \eqref{eq:ckj_rem}, we can get 
{\small
    \begin{align}\label{ineq:rhon}
    &2c^{(n)}_{n-1}+2c^{(n)}_n-2\beta_n-a_n^{(n)}
    =c^{(n)}_{n-1}+2c^{(n)}_n-a^{(n)}_{n}+c^{(n)}_{n-1}-c^{(n)}_{n-2}+a^{(n)}_{n-1}\\
    =&c^{(n)}_{n-1}+\frac{2}{(1-\alpha)\tau_n^\alpha}+\frac{2\alpha\tau_n}{(2-\alpha)(1-\alpha)(\tau_{n-1}+\tau_{n}) \tau_n^\alpha}-\frac{\alpha\tau_n^2}{(2-\alpha)(1-\alpha)\tau_{n-1}(\tau_{n-1}+\tau_{n}) \tau_n^\alpha}\nonumber\\
    &+\frac{\alpha\tau_{n-1}^3}{\tau_{n}(\tau_{n-1}+\tau_{n})}\int_0^{1} s(1-s)(t_n-t_{n-1}+s\tau_{n-1})^{-\alpha-1}\  \ds\nonumber\\
    &-\frac{\alpha\tau_{n-2}^3}{\tau_{n-1}(\tau_{n-2}+\tau_{n-1})}\int_0^{1} s(1-s)(t_n-t_{n-2}+s\tau_{n-2})^{-\alpha-1}\  \ds\nonumber\\
    & -\tau_n^{-\alpha}+\frac{\alpha\tau_{n-1}}{\tau_{n-1}+\tau_{n}}
    \int_0^{1} (\tau_{n-1}+\tau_{n}+s\tau_{n-1})
    (1-s)(t_n-t_{n-1}+s\tau_{n-1})^{-\alpha-1}\,\ds\nonumber\\
     >&c^{(n)}_{n-1}+\frac{\alpha}{(2-\alpha)(1-\alpha)\tau_{n}^{\alpha}}\bigg(
    (1+\alpha)(2-\alpha)/\alpha
  -\frac{\rho_n(\rho_n-2)}{1+\rho_n}\bigg),\nonumber
    \end{align}}
if \eqref{condition:rho} holds for $j=n-1$.
The proof of the last inequality in \eqref{ineq:rhon} is similar to the proof of property (P9) in Lemma \ref{lemmapro}.
To ensure 
$
    2c^{(n)}_{n-1}+2c^{(n)}_n-2\beta_n-a_n^{(n)}\ge 0,
$
it is sufficient to impose
\begin{align*}
&\frac{1}{\rho_{n}}\ge\frac{1}{\rho_{n-1}^2(1+\rho_{n-1})}-3,\quad   (1+\alpha)(2-\alpha)/\alpha
  -\frac{\rho_n(\rho_n-2)}{1+\rho_n}\ge0, \quad \forall~\alpha\in (0,1),
\end{align*}
that is,
\begin{equation}\label{eq:3.80}
\begin{aligned}
&\frac{1}{\rho_{n}}\ge\frac{1}{\rho_{n-1}^2(1+\rho_{n-1})}-3,\quad \rho_n\le 2+\sqrt{6}.
 \end{aligned}    
\end{equation}
Combining the above discussions on $\mathbf D$, we conclude that if \eqref{rhokk1} and  \eqref{eq:3.80} hold, then $\mathbf D$ is diagonally dominant and positive semidefinite, satisfying 
\begin{equation}
\begin{aligned}
    \mathbf D\geq (2-\alpha)^{-1}(1-\alpha)^{-1}{\rm diag}\left(0,
    g_3(\alpha),\ldots,g_k(\alpha),\ldots,g_{n}(\alpha)\right),
\end{aligned}
\end{equation}
where $g_k(\alpha)$ is given in \eqref{eq:gk}.

We now combine all the conditions for the positive semidefiniteness of $\mathbf A+\mathbf A^{\rm T}$, $\mathbf C$ and $\mathbf D$, so that
\begin{equation}
\begin{aligned}
    \mathbf M+\mathbf M^{\rm T} &= (\mathbf A+\mathbf A^{\rm T})+(\mathbf B+\mathbf B^{\rm T})\geq \mathbf B+\mathbf B^{\rm T}\\
    &\geq (2-\alpha)^{-1}(1-\alpha)^{-1}{\rm diag}\left(g_1(\alpha),g_2(\alpha),\ldots,g_k(\alpha),\ldots,g_{n}(\alpha)\right)
\end{aligned}
\end{equation}
is positive definite,
where $g_k(\alpha)$ is given in \eqref{eq:gk}.
This gives
\begin{equation*}
    \mathcal B_n(u,u) =\frac{1}{\Gamma(1-\alpha)}\int_{\Omega}\Psi \mathbf M \Psi^{\mathrm{T}}\dx\geq \sum_{k=1}^{n} \frac{g_k(\alpha)}{2\Gamma(3-\alpha)} \|\delta_k u\|^2_{L^2(\Omega)}\geq 0.
\end{equation*}
In fact, we have proved the following results:
\begin{itemize}
    \item Positive semidefiniteness of $\mathbf A+\mathbf A^{\rm T}$: \eqref{condition:rho} holds.
    \item Positive definiteness of $\mathbf C$:  $0.3\leq \rho_3\leq 4.5$ and \eqref{condc11} holds.
    \item Positive semidefiniteness of $\mathbf D$: \eqref{rhokk1} holds for $3\leq k\leq n-1$ and  \eqref{eq:3.80} holds for $k=n$.
\end{itemize}

In the following content, we just simplify the above constraints for the positive semidefiniteness of $\mathbf M+\mathbf M^{\rm T}$. 

The condition \eqref{condition:rho} actually says that $(\rho_k,\rho_{k+1})$ lies on the right-hand side of the blue solid curve in Figure \ref{fig:rho}.
Let $\rho_*\approx 0.356341$ be the positive root of $\rho(1+\rho)=1-3\rho^2(1+\rho).$
It can be found that if $\rho_{k}\leq  \rho_*$ for some $k$, then $\rho_*\geq\rho_k\geq \rho_{k+1}\geq\rho_{k+2}\geq \ldots$ and $\tau_k$ will shrink to $0$ quickly as $k$ increases.
This doesn't make sense in practice. We shall impose 
$
    \rho_{k}>\rho_*,~ \forall k\geq 2.
$
As a consequence, we have the following constraints: for $k\geq 2$,
\begin{equation}\label{condrhorho}
\left\{
\begin{aligned}
     &\rho_*<\rho_{k+1} \leq  \frac{\rho_k^2(1+\rho_k)}{1-3\rho_k^2(1+\rho_k)},&&  \rho_*<\rho_k< \eta_1,\\
      &\rho_*<\rho_{k+1},&& \eta_1\le \rho_k,
\end{aligned}
\right.
\end{equation}
where 
$ \eta_1\approx0.475329$
be the unique positive root of 
$1-3\rho^2(1+\rho)=0.$
 \begin{figure}
    \centering
    \includegraphics[width=0.5\textwidth]{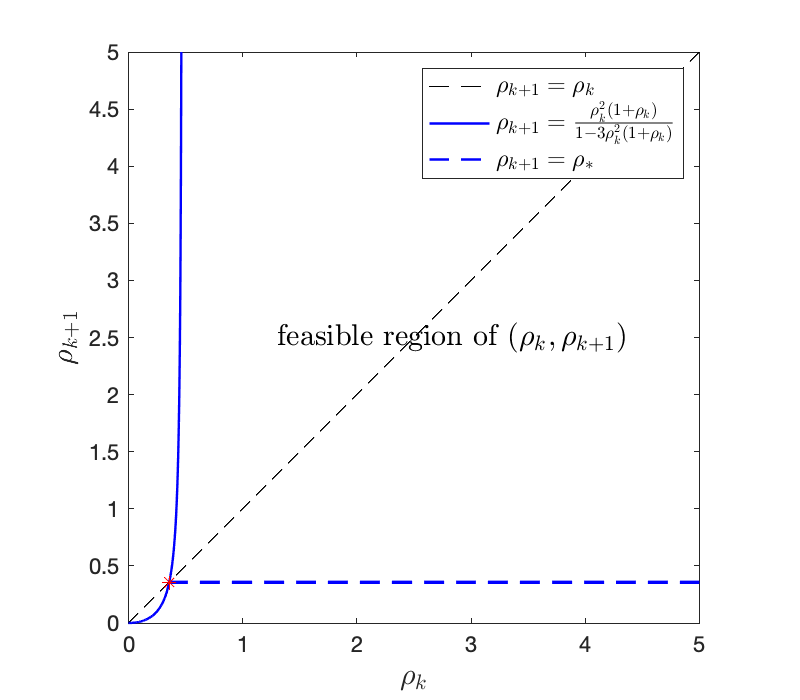}
    \caption{Feasible region of $(\rho_k,\rho_{k+1})$, enclosed by the blue solid curve and the blue dashed line, obtained from the constraint \eqref{condrhorho} for $k\geq 2$. The red star marker is point $(\rho_*,\rho_*)$.}
    \vspace{-0.2in}
    \label{fig:rho}
\end{figure}

Since $\rho_{k+1}>\rho_*$, we can obtain from the last inequality of \eqref{rhokk1}: $ \forall3\le k\le n-1$
\begin{equation}\label{eq:rho^*}
    \frac{2+\rho_*}{1+\rho_*}-\frac{\rho_k(\rho_k-2)}{1+\rho_k}>0\quad\Rightarrow\quad
    \rho_k<\rho^*\coloneqq \frac12\left[\frac{4+3\rho_*}{1+\rho_*}+\sqrt{\left(\frac{4+3\rho_*}{1+\rho_*}\right)^2+4\frac{2+\rho_*}{1+\rho_*}}\right].
\end{equation}
Then the last inequality of
\eqref{rhokk1} gives for $3\leq k\leq n-1$,
\begin{equation}\label{cond2}
\left\{
    \begin{aligned}
    &\rho_*<\rho_{k+1}< \rho^*,&& \rho_*<\rho_k<\rho^*,\\
 &\rho_*<\rho_{k+1}\leq\frac{-\rho_k^2+4\rho_k+2}{\rho_k^2-3\rho_k-1}, &&\eta_2<\rho_k<\rho^*, 
    \end{aligned}
\right.
\end{equation}
where 
$
    \eta_2=\frac{3+\sqrt{13}}{2}
$
is the positive root of 
$
    \rho^2-3\rho-1=0.
$
Note that $\rho^*\approx4.155358\le2+\sqrt{6}$, \eqref{cond2} implies $\rho_n\leq 2+\sqrt{6}$, the constraint in \eqref{eq:3.80}.

Combining \eqref{condc11}, \eqref{condrhorho} and \eqref{cond2}, it is sufficient to impose \eqref{thm1cond1} and \eqref{thm1cond2}  to ensure the positive semidefinteness of $\mathcal B_n$.
\end{proof}

\begin{corollary}\label{cor}
Let 
\begin{equation}
   \rho_{\rm L}
    \approx 0.457332766746115, \quad\rho_{\rm R} =\frac{3+\sqrt{17}}{2}\approx3.561552812808830.
\end{equation}
If $\rho_k\in [\rho_{\rm L},\rho_{\rm R}]$ for all $k\geq 2$, then
for any function $u$ defined on $[0,\infty)\times\Omega$ and $n\geq 2$,
\begin{equation}\label{eq:corcor}
    \mathcal B_n(u,u) = \sum_{k=1}^{n}\langle L_k^\alpha u, \delta_k u\rangle\geq \sum_{k=1}^{n} \frac{g_k(\alpha)}{2\Gamma(3-\alpha)} \|\delta_k u\|^2\geq C \sum_{k=1}^{n} \tau_k^{-\alpha} \|\delta_k u\|^2\geq 0,
\end{equation}
with $g_k(\alpha)$ given in \eqref{eq:gk} and $C>0$ is some constant depending on $\alpha$. 
\end{corollary}

 \begin{figure}[!ht]
    \centering
    \includegraphics[width=0.5\textwidth]{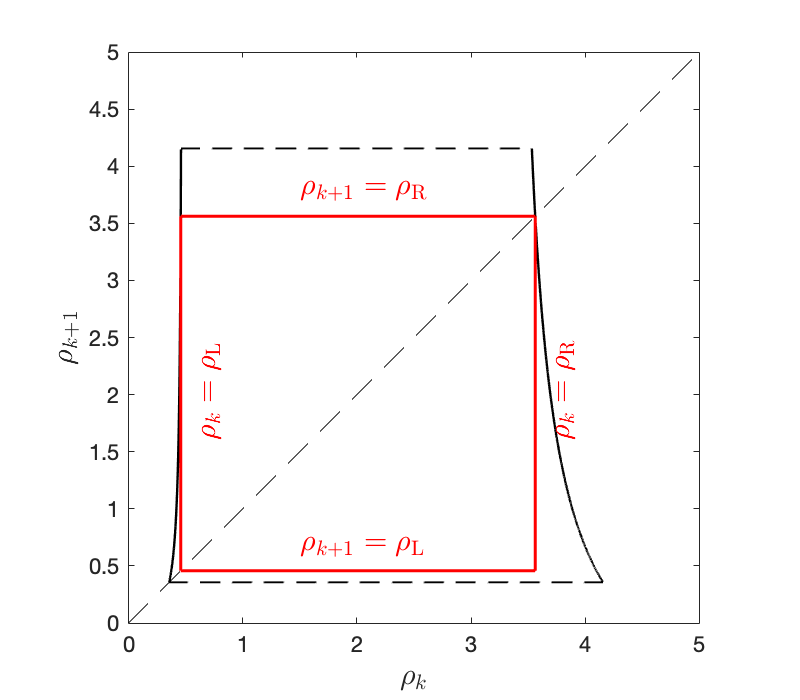}
    \caption{Region of $[\rho_{\rm L},\rho_{\rm R}]^2$ given in Corollary \ref{cor} for all $k\geq3$, which is a subregion of the region in Figure \ref{fig:rhok}.}
    \label{fig:rhok2}
\end{figure}

\begin{proof}
    For $k\ge 3$, we want to find the largest square subregion contained by the region shown in Figure \ref{fig:rhok}. In Figure \ref{fig:rhok2} we draw this square $[\rho_{\rm L}, \rho_{\rm R}]^2$.
  We first set $\rho_{k+1}=\rho_k$ in the inequality \eqref{rhokk1}. 
  Precisely speaking, we derive the quantities of $\rho_{\rm R}$ and $\rho_{\rm L}$ as follows:
$
   \rho_{\rm R} 
$
is the positive root of 
$
    1+\frac{1}{1+\rho}-\frac{\rho(\rho-2)}{1+\rho}=0
$
and 
$
  \rho_{\rm L} 
$
is the positive root of
$
    \frac{\rho^2(1+\rho)}{1-3\rho^2(1+\rho)}=\rho_{\rm R}.
$
Clearly, for any $\rho_k\in[ \rho_{\rm L}, \rho_{\rm R}]$, $k\ge 3$, the condition  \eqref{thm1cond2} in Theorem \ref{thm1} is satisfied.
Moreover, it is easy to check that if $(\rho_2,\rho_3)\in[ \rho_{\rm L}, \rho_{\rm R}]^2$, the condition \eqref{thm1cond1}  in Theorem \ref{thm1} also holds. 
\end{proof}

\begin{remark}\label{rem:3.4}
In comparison to the mesh requirements in \cite[Theorem 3.2]{kopteva2021error}, our requirements in Theorem \ref{thm1} and Corollary \ref{cor} are more flexible.
In \cite[Theorem 3.2]{kopteva2021error}, it is required that $\rho_j\geq \rho_{j+1}\geq 1$ for all $j\geq 2$, and there exists a constant $K>0$ such that $\rho_j\leq \bar\rho_\theta$ for all $j\geq K+1$, where $\bar \rho_\theta$ is the largest time step ratio depending on a parameter $\theta\in [1/2,1]$. This means that the time steps cannot decrease, and the range of allowable time step ratios is limited by $\bar\rho_\theta$. Moreover, the value of $\bar\rho_\theta$ depends on the parameter $\theta$ and the value of $\alpha$: when $\alpha = 0.1,~\bar\rho_{\frac12} = 1.0381$ and $\bar\rho_1 = 1.0849$; when $\alpha = 0.9, ~\bar\rho_{\frac12} = 1.0389$ and $\bar\rho_1 = 1.8015$.

In contrast, Corollary \ref{cor} only requires that the time step ratios are in the range $[0.4573, 3.5615]$, which allows for both increasing and decreasing time steps. This flexibility is important when using adaptive time meshes. 
\end{remark}

\section{$H^1$-stability of L2-type method for subdiffusion equation}\label{sect4}
\subsection{Stability for general nonuniform meshes}
We consider the following subdiffusion equation:
\begin{equation}\label{eq:subdiffusion}
\begin{aligned}
    \partial_t^\alpha u(t,x) & = \Delta u(t,x)+f(t,x),&& (t,x)\in (0,\infty)\times\Omega,\\
    u(t,x) &= 0,&& (t,x)\in (0,\infty)\times\partial\Omega,\\
    u(0,x)& = u^0(x),&& x\in \Omega,
\end{aligned}    
\end{equation}
where $\Omega$ is a bounded Lipschitz domain in $\mathbb R^d$. 
Given an arbitrary nonuniform mesh $\{\tau_k\}_{k\ge1}$, the L2 scheme of this subdiffusion equation is written as
\begin{equation}\label{eq:sch_sub}
\begin{aligned}
    L_k^\alpha u 
    &= \Delta u^k+f^k,&&  \text{in}~ \Omega,\\
    u^k&=0,&&\text{on} ~\partial\Omega,
\end{aligned}    
\end{equation}
where $f^k=f(t_k,\cdot)$.
\begin{theorem}\label{thm2}
Assume that $f(t,x) \in L^\infty([0,\infty);L^2(\Omega)) \cap BV([0,\infty); L^2(\Omega))$ is a bounded variation function in time and $ u^0\in H_0^1(\Omega)$.
If the nonuniform mesh $\{\tau_k\}_{k\ge1}$ satisfies \eqref{thm1cond1} and \eqref{thm1cond2} (or simply $\rho_k\in[\rho_{\rm L},\rho_{\rm R}]$ given in Corollary \ref{cor}),
then the numerical solution $u^n$ of the L2 scheme \eqref{eq:sch_sub} satisfies the following $H^1$-stability
\begin{equation}
\begin{aligned}
    \|\nabla u^n\|_{L^2(\Omega)}
   &\le\|\nabla u^0\|_{L^2(\Omega)} +2C_{f}C_{\Omega},
\end{aligned}
\end{equation}
where $C_f$ depends on the source term $f$, $C_\Omega$ is the Sobolev embedding constant depending on $\Omega$ and the dimension $d$.
\end{theorem}
\begin{proof}
When $n=1$, we have 
\begin{equation}\label{casen1}
    \frac{\delta_1 u}{\Gamma(2-\alpha)\tau_1^\alpha}=\Delta u^1+f^1.
\end{equation}
Multiplying \eqref{casen1} with $\delta_1 u$ and integrating over $\Omega$ yield 
\begin{equation*}
     \frac{\|\delta_1 u\|^2_{L^2(\Omega)}}{\Gamma(2-\alpha)\tau_1^\alpha}=-\frac12 \|\nabla u^1\|_{L^2(\Omega)}^2+ \frac12 \|\nabla u^0\|_{L^2(\Omega)}^2 - \frac12 \|\nabla \delta_1 u\|_{L^2(\Omega)}^2+\langle f^1, \delta_1 u \rangle.
\end{equation*}
Applying Cauchy--Schwarz inequality, then we can derive
\begin{align}\label{case1}
    \|\nabla u^1\|_{L^2(\Omega)}^2\le&\|\nabla u^0\|_{L^2(\Omega)}^2+4\|f\|_{L^\infty([0,\infty);L^2(\Omega))}\max_{0\le k\le 1} {\| u^k\|_{L^2(\Omega)}}\\
     \le &\|\nabla u^0\|_{L^2(\Omega)}^2+4\|f\|_{L^\infty([0,\infty);L^2(\Omega))}C_{\Omega}  \max_{0\le k\le 1} \|\nabla u^k\|_{L^2(\Omega)},
\end{align}
where $C_{\Omega}$ is the Sobolev embedding constant depending on $\Omega$ and the dimension.

We now consider the case  $n\ge 2$.
Multiplying \eqref{eq:sch_sub} with $\delta_k u$, integrating over $\Omega$, and summing up the derived equations over $n$ yield
\begin{align*}
\sum_{k=1}^{n}\langle L_k^\alpha u, \delta_k u\rangle
    =& \sum_{k=1}^{n}\langle \Delta u^k, \delta_k u\rangle +\sum_{k=1}^{n}\langle f^k, \delta_k u \rangle\\
     = & -\frac12 \|\nabla u^n\|_{L^2(\Omega)}^2 + \frac12 \|\nabla u^0\|_{L^2(\Omega)}^2 - \frac12 \sum_{k=1}^{n} \|\nabla \delta_k u\|_{L^2(\Omega)}^2\\
     &+\langle f^n,  u^n\rangle-\langle f^1,  u^0\rangle-\sum_{k=2}^{n}\langle \delta_k f, u^{k-1}\rangle.\end{align*}    
Applying the Cauchy--Schwarz inequality gives
\begin{equation*}
\begin{aligned}
    &\langle f^n,  u^n\rangle-\langle f^1,  u^0\rangle+\sum_{k=2}^{n}\langle \delta_k f,  u^{k-1}\rangle\\
    \le &\left( 2\|f\|_{L^\infty([0,\infty);L^2(\Omega))}+\|f\|_{BV([0,\infty); L^2(\Omega))}\right) \max_{0\le k\le n} {\|u^k\|_{L^2(\Omega)}}\\
     \le &C_{f}C_{\Omega}  \max_{0\le k\le n} {\|\nabla u^k\|_{L^2(\Omega)}},
\end{aligned} 
\end{equation*}
where $C_{f}= 2\|f\|_{L^\infty([0,\infty);L^2(\Omega))}+\|f\|_{BV([0,\infty); L^2(\Omega))}$. From Theorem \ref{thm1}, we then have for $n\ge2$,
\begin{equation}\label{case2}
\begin{aligned}
\|\nabla u^n\|_{L^2(\Omega)}^2 
     \le &\|\nabla u^0\|_{L^2(\Omega)}^2+2C_{f}C_{\Omega}  \max_{0\le k\le n} {\|\nabla u^k\|_{L^2(\Omega)}}.
     \end{aligned}    
\end{equation}
Note that \eqref{case1} implies that \eqref{case2} also holds for $n=1$.
For any $N \ge 1$, we take $\max_{0\le n\le N}$ on both sides of \eqref{case2}, to obtain
\begin{equation*}
\begin{aligned}
\max_{0\le n\le N}\|\nabla u^n\|_{L^2(\Omega)}^2 
     \le  \|\nabla u^0\|_{L^2(\Omega)}^2+2C_{f}C_{\Omega}  \max_{0\le n\le N} {\|\nabla u^n\|_{L^2(\Omega)}},
     \end{aligned}    
\end{equation*}
which indicates 
\begin{equation*}
\begin{aligned}
    \max_{0\le n\le N}\|\nabla u^n\|_{L^2(\Omega)}\le C_{f}C_{\Omega} +\sqrt{(C_{f}C_{\Omega})^2+\|\nabla u^0\|_{L^2(\Omega)}^2}\le\|\nabla u^0\|_{L^2(\Omega)} +2C_{f}C_{\Omega}.
\end{aligned}
\end{equation*}
\end{proof}

\begin{remark}
In \cite{quan2020numerical,quan2022energy}, it is proved that the L1 scheme on an arbitrary nonuniform mesh and the L2 scheme on uniform meshes are energy stable for time-fractional gradient flows, where the source term $f$ depends on $u$.
\end{remark}

\begin{remark}
In \cite{liao2021analysis}, Liao-Zhang consider the BDF$2$ scheme with nonuniform meshes for the  diffusion equation ($\alpha=1$) and prove that the scheme is stable if $\rho_k\leq (3+\sqrt{17})/2$.
Their energy stability result is similar to Theorem \ref{thm2}, but for integer-order diffusion equation.
\end{remark}

\subsection{Stability for graded meshes}
Consider the subdiffusion equation in finite time:
\begin{equation}\label{eq:subdiffusion1}
\begin{aligned}
    \partial_t^\alpha u(t,x) & = \Delta u(t,x)+f(t,x),&& (t,x)\in (0,T]\times\Omega,\\
    u(t,x) &= 0,&& (t,x)\in (0,T]\times\partial\Omega,\\
    u(0,x)& = u^0(x),&& x\in \Omega.
\end{aligned}    
\end{equation}
The L2 scheme of the subdiffusion equation is still written as
\begin{equation}\label{eq:sch_sub_graded}
\begin{aligned}
    L_k^\alpha u 
    &= \Delta u^k+f^k \quad\mbox{with}\quad f^k=f(t_k,\cdot).
\end{aligned}    
\end{equation}

The graded mesh with grading parameter $r>1$ is given by
\begin{equation}\label{eq:tauj}
\begin{aligned}
    t_j = \left(\frac j N\right)^r T, \quad \tau_j = t_j -t_{j-1} = \left[\left(\frac j N\right)^r-\left(\frac {j-1} N\right)^r\right]  T.
\end{aligned}
\end{equation}
Recall that the constraint \eqref{thm1cond1} for $k=2$ in Theorem \ref{thm1} is
 \begin{equation*}
     2+\frac{2}{1+\rho_3}+\frac{4\rho_2}{1+\rho_2}-\frac{\rho_2^3}{(1+\rho_2)^2}\ge 0,
 \end{equation*}
which gives a restriction on $r$ for the graded mesh \eqref{eq:tauj},
\begin{equation}\label{constraint_r0}
    1<r\leq 3.1253645.
\end{equation}
Moreover, it is easy to check that if \eqref{constraint_r0} is satisfied, $\rho_3\in[\rho_{\rm L},\rho_{\rm R}]$. Since $\rho_k$ decreases w.r.t. $k\ge 2$,
all constraints in Theorem \ref{thm1} are satisfied when \eqref{constraint_r0} holds. 
Therefore, the $H^1$-stability can be established if $1<r\leq 3.1253645$ according to Theorem \ref{thm2}.

However, we can provide an even better result on the constraint of $r$ by improving the splitting of $\mathbf B+\mathbf B^{\rm T}$ in the proof of Theorem \ref{thm1}.

\begin{theorem}\label{thm3}
Assume that $f(t,x) \in L^\infty([0,T];L^2(\Omega)) \cap BV([0,T]; L^2(\Omega))$ is a bounded variation function in time and $ u^0\in H_0^1(\Omega)$.
If the graded mesh defined by \eqref{eq:tauj} satisfies 
$
    1< r\leq 3.2016538,
$
then the numerical solution $u^n$ of the L2 scheme \eqref{eq:sch_sub_graded} satisfies the following $H^1$-stability:
\begin{equation}
\begin{aligned}
    \|\nabla u^n\|_{L^2(\Omega)}
   &\le\|\nabla u^0\|_{L^2(\Omega)} +2C_{f}C_{\Omega},
\end{aligned}
\end{equation}
where $C_f$ depends on the source term $f$, $C_\Omega$ is the Sobolev embedding constant depending on $\Omega$ and the dimension $d$.
\end{theorem}

\begin{proof}
We only need to prove the positive semidefiniteness of 
$\mathcal B_n(u,u)$ for the graded mesh.
As in the proof of Theorem \ref{thm1}, $\mathbf A + \mathbf A^{\rm T}$ is positive semidefinite for the graded mesh due to $\rho_k>1$.
Now we consider the following splitting
\begin{equation*}
    \mathbf B + {\mathbf B}^{\rm T}=
    \begin{pmatrix}
\mathbf C_0& \mathbf 0\\
\mathbf 0 &  \mathbf 0\\
\end{pmatrix}_{ n\times n}+
\begin{pmatrix}
\mathbf 0& \mathbf 0\\
\mathbf 0 &  \mathbf D_0\\
\end{pmatrix}_{n\times n},
\end{equation*}
where 
\begin{equation*}
\mathbf C_0=\begin{pmatrix}
\mathbf C_1&\mathbf 0\\
\mathbf 0 &  \mathbf 0\\
\end{pmatrix}_{5\times5}+
\begin{pmatrix}
 0&\mathbf 0&\mathbf 0\\
\mathbf 0&\mathbf C_2&\mathbf 0\\
\mathbf 0 &  \mathbf 0&\mathbf 0\\
\end{pmatrix}_{5\times5}
+\begin{pmatrix}
\mathbf 0&\mathbf 0&\mathbf 0\\
\mathbf 0&\mathbf C_3&\mathbf 0\\
\mathbf 0 &  \mathbf 0& 0\\
\end{pmatrix}_{5\times5}
+\begin{pmatrix}
\mathbf 0&\mathbf 0\\
\mathbf 0 &  \mathbf C_4\\
\end{pmatrix}_{5\times5},
\end{equation*}
with 
 \begin{equation*}
 \begin{aligned}
\mathbf C_1&=\begin{pmatrix}
2(1-\alpha)^{-1}\tau_1^{-\alpha}-2\beta_1& -a_2^{(2)}\\
-a_2^{(2)} &  2c_1^{(2)}+2c_2^{(2)}-2\beta_2-0.7013a_3^{(3)}\\
\end{pmatrix}_{2\times2},\\
\mathbf C_2&=\begin{pmatrix}
0.7013a_3^{(3)}&-a_3^{(3)}\\
-a_3^{(3)} &  2c_2^{(3)}+2c_3^{(3)}-2\beta_3-0.45473a^{4}_4\\
\end{pmatrix}_{2\times2},\\
\mathbf C_3&=\begin{pmatrix}
0.45473a^{(4)}_4&-a^{(4)}_4\\
-a^{(4)}_4 &  2c_3^{(4)}+2c^{(4)}_4-2\beta_4-0.4131a^{(5)}_5\\
\end{pmatrix}_{2\times2},\\
\mathbf C_4&=\begin{pmatrix}
0.4131a^{(5)}_5&-a^{(5)}_5\\
-a^{(5)}_5 &  2c^{(5)}_4+2c^{(5)}_5-2\beta_5-a^{(6)}_6\\
\end{pmatrix}_{2\times2},
 \end{aligned}
\end{equation*}
and
\begin{equation*}
\footnotesize
\begin{aligned}
    &\mathbf D_0=
 \begin{pmatrix}
 a^{(6)}_6&-a^{(6)}_6\\
  -a^{(6)}_6& 2c_5^{(6)}+2c_6^{(6)}-2\beta_6&-a^{(7)}_7\\
  &  \ddots & \ddots&\ddots\\
 & &  -a_n^{(n)}& 2c^{(n)}_{n-1}+2c^{(n)}_n-2\beta_n
\end{pmatrix}_{(n-4)\times (n-4)}.
\end{aligned}
\end{equation*}
Here, we consider the case of $n\geq 5$ by default, while in the case of $n\leq 4$, the proof is even simpler. 

We now study the range of $r$ to ensure the  positive semidefiniteness of $\mathbf C_1$, $\mathbf C_2$, $\mathbf C_3$, and $\mathbf C_4$.  
Note that from \eqref{tau12}, we have 
\begin{align}\label{eq:c1_11}
    [\mathbf C_1]_{11} =[\mathbf C]_{11} 
   >\frac{1+\alpha}{(1-\alpha)\tau_1^{\alpha}}>0.
\end{align}
Similar as \eqref{eq:ck2}, we have the following inequalities
\begin{equation}\label{eq:c1_22}
\small
\begin{aligned}
     [\mathbf C_1]_{22} 
     &>\frac{\alpha}{(2-\alpha)(1-\alpha)\tau_{2}^{\alpha}}\bigg(
    \frac{(1+\alpha)(2-\alpha)}{\alpha}+\frac{\rho_{3}-1+\alpha}{(1+\rho_{3})^{\alpha}}-\frac{1.7013\rho_{3}^{2-\alpha}}{1+\rho_{3}}+\frac{2\rho_2}{1+\rho_2}\bigg),\\
     [\mathbf C_2]_{22} 
     &>\frac{\alpha}{(2-\alpha)(1-\alpha)\tau_{3}^{\alpha}}\bigg(
    \frac{(1+\alpha)(2-\alpha)}{\alpha}+\frac{\rho_{4}-1+\alpha}{(1+\rho_{4})^{\alpha}}-\frac{1.45473\rho_{4}^{2-\alpha}}{1+\rho_{4}}+\frac{2\rho_3}{1+\rho_3}\bigg),\\
    [\mathbf C_3]_{22} 
    & >\frac{\alpha}{(2-\alpha)(1-\alpha)\tau_{4}^{\alpha}}\bigg(
    \frac{(1+\alpha)(2-\alpha)}{\alpha}+\frac{\rho_{5}-1+\alpha}{(1+\rho_{5})^{\alpha}}-\frac{1.4131\rho_{5}^{2-\alpha}}{1+\rho_{5}}+\frac{2\rho_4}{1+\rho_4}\bigg),\\
     [\mathbf C_4]_{22} 
     &>\frac{\alpha}{(2-\alpha)(1-\alpha)\tau_{5}^{\alpha}}\bigg(
    \frac{(1+\alpha)(2-\alpha)}{\alpha}+\frac{\rho_{6}-1+\alpha}{(1+\rho_{6})^{\alpha}}-\frac{2\rho_{6}^{2-\alpha}}{1+\rho_{6}}+\frac{2\rho_5}{1+\rho_5}\bigg).
\end{aligned}
\end{equation}
From \eqref{eq:a_k1}, \eqref{eq:c1_11} and \eqref{eq:c1_22}, we have
\begin{align*}
\small
         [\mathbf C_1]_{11}  [\mathbf C_1]_{22} -[\mathbf C_1]_{12} [\mathbf C_1]_{21}
     &> \frac{\alpha^2}{(1-\alpha)^2(2-\alpha)^2\tau_2^{2\alpha}} \kappa_1(\alpha),\\
         [\mathbf C_2]_{11}  [\mathbf C_2]_{22} -[\mathbf C_2]_{12} [\mathbf C_2]_{21}
     &> \frac{0.7013\alpha a_3^{(3)}}{(1-\alpha)(2-\alpha)\tau_3^{\alpha}} \kappa_2(\alpha),\\
         [\mathbf C_3]_{11}  [\mathbf C_3]_{22} -[\mathbf C_3]_{12} [\mathbf C_3]_{21}
     &> \frac{0.45473\alpha a_4^{(4)}}{(1-\alpha)(2-\alpha)\tau_4^{\alpha}} \kappa_3(\alpha),\\
         [\mathbf C_4]_{11}  [\mathbf C_4]_{22} -[\mathbf C_4]_{12} [\mathbf C_4]_{21}
     &> \frac{0.4131\alpha a_5^{(5)}}{(1-\alpha)(2-\alpha)\tau_5^{\alpha}} \kappa_4(\alpha),\\
\end{align*}
where
\begin{equation*}
\footnotesize
 \begin{aligned}
\kappa_1(\alpha)=&\frac{(1+\alpha)(2-\alpha)\rho_2^\alpha}{\alpha}\bigg(
    \frac{(1+\alpha)(2-\alpha)}{\alpha} +\frac{\rho_{3}-1+\alpha}{(1+\rho_{3})^{\alpha}}
   -\frac{1.7013\rho_{3}^{2-\alpha}}{1+\rho_{3}} +\frac{2\rho_2}{1+\rho_2}\bigg)-\left(\frac{\rho_2^2}{1+\rho_2}\right)^2,\\
   \kappa_2(\alpha)=&
    \frac{(1+\alpha)(2-\alpha)}{\alpha}+\frac{\rho_{4}-1+\alpha}{(1+\rho_{4})^{\alpha}}-\frac{1.45473\rho_{4}^{2-\alpha}}{1+\rho_{4}}+\frac{2\rho_3}{1+\rho_3}-\frac{\rho_3^2}{0.7013(1+\rho_3)}, \\
     \kappa_3(\alpha)=&
    \frac{(1+\alpha)(2-\alpha)}{\alpha}+\frac{\rho_{5}-1+\alpha}{(1+\rho_{5})^{\alpha}}-\frac{1.4131\rho_{5}^{2-\alpha}}{1+\rho_{5}}+\frac{2\rho_4}{1+\rho_4}-\frac{\rho_4^2}{0.45473(1+\rho_4)},\\
    \kappa_4(\alpha)=&
    \frac{(1+\alpha)(2-\alpha)}{\alpha}+\frac{\rho_{6}-1+\alpha}{(1+\rho_{6})^{\alpha}}-\frac{2\rho_{6}^{2-\alpha}}{1+\rho_{6}}+\frac{2\rho_5}{1+\rho_5}-\frac{\rho_5^2}{0.4131(1+\rho_5)}.
    \end{aligned}
 \end{equation*}
 Here for the graded mesh with grading parameter $r>1$, 
 \begin{equation*}
     \rho_k = \frac{k^r-(k-1)^r}{(k-1)^r-(k-2)^r},\quad k\geq 2,
 \end{equation*}
depends only on $k$ and $r$.
In Figure \ref{fig:kappa}, we illustrate $\kappa_1'(\alpha)$, $\kappa_2'(\alpha)$, $\kappa_3'(\alpha)$, $\kappa_4'(\alpha)$ w.r.t. $\alpha \in (0,1)$ and $r\in [1,3.25]$ for the graded mesh.
It can be observed that $\kappa_i'(\alpha)\leq 0, ~i=1,2,3,4$ for $\alpha\in(0,1)$ and $r\in[1,3.25]$.
A more rigorous proof can be provided but is omitted here due to the length of this work. 
Thus for any fixed $r\in[1,3.25]$,  $\kappa_i$ decreases w.r.t. $\alpha \in (0,1)$. 
  \begin{figure}[!ht]
    \centering
    \includegraphics[trim={0 4.in 0.5in 0},clip,width=0.48\textwidth]
    {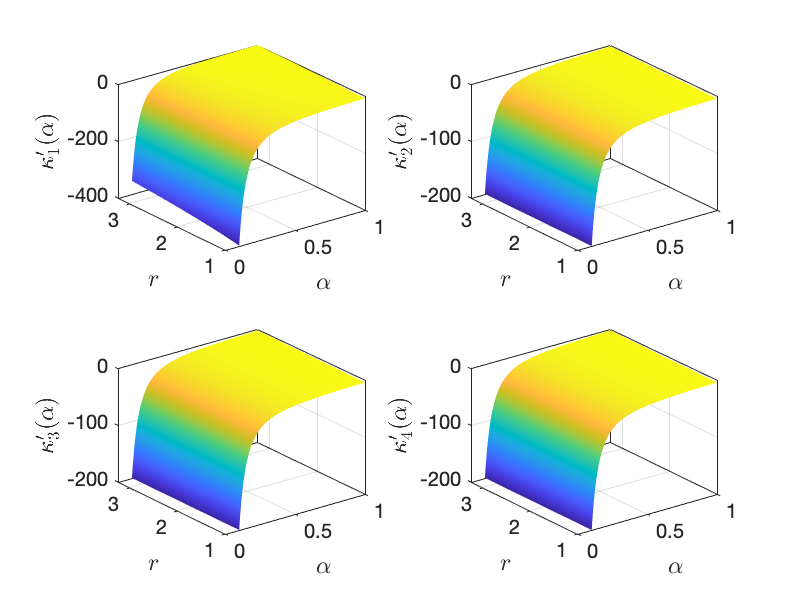}
    \includegraphics[trim={0.5in 0 0 4.in},clip,width=0.48\textwidth]
    {kappa.png}
    \caption{$\kappa_1'(\alpha)$, $\kappa_2'(\alpha)$, $\kappa_3'(\alpha)$, $\kappa_4'(\alpha)$ w.r.t. $(\alpha,r)$ where $\alpha\in(0,1)$ and $r\in[1,3.25]$. }
    \label{fig:kappa}
\end{figure}

To ensure $\kappa_1(\alpha)\geq 0$, we need to impose
$\kappa_1(1)\ge 0$, i.e.,
\begin{equation*}
     \rho_2^{-1}\kappa_1(1)=4-\frac{1.4026\rho_3}{1+\rho_3}+\frac{4\rho_2}{1+\rho_2}-\frac{\rho_2^3}{(1+\rho_2)^2}\ge 0,
\end{equation*}
which results in
\begin{equation}\label{eq:5.19}
   1< r\le 3.2016538.
\end{equation}
Further, when \eqref{eq:5.19} holds, it is easy to verify that $\kappa_i(1)\ge 0$, $i=2,3,4$ implying that $\kappa_i(\alpha)\ge 0$ for all $\alpha\in(0,1)$.
Since $[\mathbf C_i]_{11}>0, ~i=1,2,3,4$, the $2\times2$ matrices $\mathbf C_i$ are positive semidefinite when \eqref{eq:5.19} holds.
 
Similar to the proof of Theorem \ref{thm1}, the positive semidefiniteness of $\mathbf D_0$ can be guaranteed because $\rho_k\in (1,\rho_6] \subset[\rho_{\rm L},\rho_{\rm R}]$ for any $k\geq 6$, when \eqref{eq:5.19} holds.
The proof is completed.
\end{proof}

 
To better understand Theorem \ref{thm3}, we do a numerical test on the matrix $\mathbf M$ defined in \eqref{matM} for the graded mesh.
We take $T =1$, $n=K=7$,  $\alpha = 0.99999$  and $r= 3.20185$ (slightly larger than $3.2016538$ in Theorem \ref{thm3}). As a consequence, the symmetric matrix $\mathbf M+\mathbf M^{\mathrm T}$ has a negative eigenvalue, i.e., $\mathbf M+\mathbf M^{\mathrm T}$ is not positive semidefinite.
 This indicates that the constraint $r\leq 3.2016538$ in Theorem \ref{thm3} is {\em almost optimal} to ensure the positive semidefiniteness of $\mathcal B_n(u,u)$.

 \section{$H^1$-convergence of L2-type method for subdiffusion equation}\label{sect5}

Let $\Pi_{1,j}$ and $\Pi_{2,j}$ be the standard Lagrange interpolation operators with interpolation  points:
 \begin{equation*}
     \Pi_{1,j}:{t_{j-1},t_j},\quad \Pi_{2,j}:{t_{j-1},t_j,t_{j+1}}.
\end{equation*}
We have the following estimate of the truncation error of L2 method.
\begin{lemma}\label{lem:truncation}
Given a function $u$ satisfying $|\partial_t^l u(t)|\le C_l(1+t^{\alpha-l})$ for $l=1,3$ and nonuniform mesh $\{\tau_k\}_{k\ge 1}$ with $\rho_k \in [\rho_{\rm min}, \rho_{\rm max}]$ for some $0<\rho_{\rm min}<\rho_{\rm max}$, the truncation error is given by 
\begin{equation}\label{eq:localrk}
r_k \coloneqq \frac{1}{\Gamma(1-\alpha)}\int _0^{t_k}(t_k-s)^{-\alpha}\partial_s[u(s)-I_2u(s)]\,\ds,\quad k\ge 1,
\end{equation} 
where $I_2 u=\Pi_{2,j}u$ on $(t_{j-1},t_j)$ for $j<k$ and $I_2 u=\Pi_{2,k-1}u$ on $(t_{k-1},t_k)$.
Then  $|r_1|\le C$ and for $k\geq 2$,
\begin{equation}\label{eq:approxerror}
   | r_k|\le C\bigg(\tau_1^\alpha( t_k-t_1)^{-\alpha}
   +\sum_{j=2}^{k-1}\tau_j^4(1+t_{j-1}^{\alpha-3})(t_k-t_j)^{-\alpha-1}+\tau_k^{3-\alpha}(1+t_{k-1}^{\alpha-3})\bigg),
\end{equation}
where $C$ is some constant depending on $\alpha,~\rho_{\rm min},~\rho_{\rm max}$ and $C_l$ for $l=1,3$.
\end{lemma}
 \begin{proof}
 Let $C$ be a generic constant depending on $\alpha,~\rho_{\rm min},~\rho_{\rm max}$ and $C_l$.
The case of $k=1$ is easy to verify. 
We now consider the case of $k\geq 2$.


 Let $\chi(s) \coloneqq u-I_2u$.  On the interval $(t_0,t_1)$, it is not difficult to obtain 
$
     |\partial_s\chi(s)|\leq  C s^{\alpha-1},
$
which yields  (see similar result in \cite[Equation (5.12)]{stynes2017error})
 \begin{equation}\label{eq:intervalt1}
  \left|\int _0^{t_1}(t_k-s)^{-\alpha}\partial_s\chi(s)\,\ds\right|
  \le C \tau_1^\alpha(t_k-t_1)^{-\alpha}.
\end{equation}
On the interval $(t_{j-1},t_j)$, $2\le j\le k-1$, we have 
\begin{equation*}
    |\chi(s)|=\left|\frac{u^{(3)}(\xi)}{6}(s-t_{j-1})(s-t_{j})(s-t_{j+1})\right|\le C\tau_j^3(1+t_{j-1}^{\alpha-3}),
\end{equation*}
where $\xi \in (t_{j-1},t_{j+1})$. 
This yields
\begin{equation}\label{eq:intervaltj}
\begin{aligned}
&\left| \int _{t_{j-1}}^{t_j}(t_k-s)^{-\alpha}\partial_s\chi(s)\,\ds\right|=
\left|-\alpha\int _{t_{j-1}}^{t_j}(t_k-s)^{-\alpha-1}\chi(s)\,\ds\right|\\
\le&C (1+t_{j-1}^{\alpha-3})\tau_j^3\int _{t_{j-1}}^{t_j}(t_k-s)^{-\alpha-1}\,\ds  
\le C\tau_j^4(1+t_{j-1}^{\alpha-3})(t_k-t_j)^{-\alpha-1}.
\end{aligned}
\end{equation}
On the interval $(t_{k-1},t_k)$, we have
\begin{equation*}
      |\chi(s)|\le C(1+t_{k-1}^{\alpha-3})(s-t_{k-2})(t_k-s) (s-t_{k-1})
    \le C\tau_k^2 (1+t_{k-1}^{\alpha-3})(t_k-s), 
\end{equation*}
which yields
\begin{equation}\label{eq:intervaltk}
\begin{aligned}
&\left| \int _{t_{k-1}}^{t_k}(t_k-s)^{-\alpha}\partial_s\chi(s)\,\ds\right|=\left|-\alpha\int _{t_{k-1}}^{t_k}(t_k-s)^{-\alpha-1}\chi(s)\,\ds\right|\\
\le&C\tau_k^2(1+t_{k-1}^{\alpha-3})\int _{t_{k-1}}^{t_k}(t_k-s)^{-\alpha}\,\ds\le C\tau_k^{3-\alpha}(1+t_{k-1}^{\alpha-3}).
\end{aligned}
\end{equation}
Combining \eqref{eq:intervalt1}, \eqref{eq:intervaltj} and \eqref{eq:intervaltk}, we obtain the estimate \eqref{eq:approxerror} of truncation error. 
 \end{proof}
 
 \begin{theorem}[$H^1$-convergence for nonuniform meshes]\label{thm:conv_nonu}
 Assume that $u\in C^3((0,T],H^1_0(\Omega))$ is the solution to \eqref{eq:subdiffusion1} and $|\partial_t^{l} u(t)|\le C_l(1+t^{\alpha-l})$ for $l=1,3$, $0< t\le T$. If the nonuniform mesh satisfies $\rho_k\in [\rho_{\rm L},\rho_{\rm R}]$, then the numerical solutions $u^k$ of L2 scheme \eqref{eq:sch_sub_graded} have the following error estimate:
 \begin{equation*}
 \begin{aligned}
    &\| \nabla u(t_n) -\nabla u^n\|_{L^2(\Omega)}^2 \le \\
    & C  \sum_{k=2}^{n} \tau_k^{\alpha}\bigg(\tau_1^\alpha( t_k-t_1)^{-\alpha}
   +\sum_{j=2}^{k-1}\tau_j^4(1+t_{j-1}^{\alpha-3})(t_k-t_j)^{-\alpha-1}+\tau_k^{3-\alpha}(1+t_{k-1}^{\alpha-3})\bigg)^2,
   \end{aligned}
 \end{equation*}
where $C$ is a constant depending on $\alpha$ and $C_l$, $l=1,3$ and $\Omega$.
 \end{theorem}
\begin{proof}
Let $e^k=u(t_k)-u^k$.
We have the error equation
\begin{equation}\label{eq:error}
\begin{aligned}
       L_k^\alpha e 
    = \Delta e^k-r_k  \text{ in}~ \Omega,\quad
    u^k =0\text{ on} ~\partial\Omega,
\end{aligned}    
\end{equation}
where $r_k$ is defined in \eqref{eq:localrk}.
Multiplying \eqref{eq:error} with $\delta_k e = e^{k}-e^{k-1}$, integrating over $\Omega$, and summing up the derived equations over $n$ yield
\begin{equation*}
\sum_{k=1}^{n}\langle L_k^\alpha e, \delta_k e\rangle
    = -\frac12 \|\nabla e^n\|_{L^2(\Omega)}^2 + \frac12 \|\nabla e^0\|_{L^2(\Omega)}^2 - \frac12 \sum_{k=1}^{n} \|\nabla \delta_k e\|_{L^2(\Omega)}^2-\sum_{k=1}^{n}\langle r_k, \delta_k e \rangle. 
\end{equation*}
From  Corollary \ref{cor}, we then have
\begin{equation}\label{case22}
\begin{aligned}
\|\nabla e^n\|_{L^2(\Omega)}^2 
     &\le -  \sum_{k=1}^{n} \|\nabla \delta_k e\|_{L^2(\Omega)}^2-C\sum_{k=1}^{n}\tau_k^{-\alpha}\| \delta_k e\|^2_{L^2(\Omega)}-2\sum_{k=1}^{n}\langle r_k, \delta_k e \rangle
\\&
     \leq C^{-1} \sum_{k=1}^{n} \tau_k^{\alpha} \|r_k\|_{L^2(\Omega)}^2.
     \end{aligned}    
\end{equation}
The desired error estimate is obtained from Lemma \ref{lem:truncation}. 
\end{proof}

In Theorem \ref{thm:conv_nonu}, it is required $\rho_{\rm L}\leq \rho_k \leq \rho_{\rm R}$. 
However, for the standard graded mesh \eqref{eq:tauj}, when $r$ is large, the ratios of first several time steps often exceed $\rho_{\rm R}$. 
This motivates researchers to modify the graded meshes. 
Let 
\begin{equation*}
    k_0 \coloneqq \min \left\{k\in \mathbb N:~ \frac{(k+2)^r-(k+1)^r}{(k+1)^r-k^r}\leq \rho_{\rm R}\right\}.
\end{equation*}
We consider the following modified graded mesh proposed by Kopteva in \cite{kopteva2021error}:
\begin{equation}\label{eq:modgm}
 t_k = T\frac{\left(k+k_0\right)^r - k_0^r}{\left(N+k_0\right)^r - k_0^r}, \quad \tau_k = t_k-t_{k-1},\quad 1\leq k\leq N.
\end{equation}
Note that when $k_0 = 0$, this modified graded mesh \eqref{eq:modgm} coincides with the standard graded mesh \eqref{eq:tauj}.
It is not difficult to find that the modified graded mesh satisfies
\begin{equation}\label{eq:prop_modgm}
\tau_k  \leq C T N^{-r} k^{r-1}\quad \mbox{and}\quad 1\leq {\tau_k}/{\tau_{k-1}}\leq \rho_{\rm R}.
\end{equation}
\begin{lemma}\label{lem:rkgraded}
For the modified graded mesh \eqref{eq:modgm}, the truncation error $r_k$ defined in \eqref{eq:localrk} satisfies
$
   | r_k|\le C k^{-\min\{r\alpha,3-\alpha\}}.
$
\end{lemma}
\begin{proof}
     Based on Lemma \ref{lem:truncation}, this proof is similar to the one of \cite[Lemma 5.2]{stynes2017error}. 
\end{proof}
\begin{theorem}[$H^1$-convergence for modified graded mesh]\label{thm:conv_graded}
Assume that $u\in C^3((0,T],H^1_0(\Omega))$ is the solution to \eqref{eq:subdiffusion1} and $|\partial_t^{l} u(t)|\le C_l(1+t^{\alpha-l})$ for $l=1,3$, $0< t\le T$. 
Consider the numerical solutions of L2 scheme \eqref{eq:sch_sub_graded} on the modified graded mesh \eqref{eq:modgm}.
For any $1\leq n\leq N$, we have the following error estimate
\begin{equation}\label{eq:error_modgm}
\| \nabla u(t_n) -\nabla u^n\|_{L^2(\Omega)} \le \left\{
\begin{aligned}
& C N^{-(2r\alpha+\alpha-1)/2}, && {\rm if~} 1< r<\max\left\{1,1/\alpha-1\right\},\\
& C N^{-r\alpha/2}, && {\rm if~}\max\left\{1,1/\alpha-1\right\}< r<5/\alpha-1,\\
& C N^{-(5-\alpha)/2}, && {\rm if~} r>5/\alpha-1.
\end{aligned}
\right.
\end{equation}
If $r= 1/\alpha-1$ with $\alpha<1/2$, or $r =5/\alpha-1$, we have
\begin{equation}\label{eq:error_modgm2}
     \| \nabla u(t_n) -\nabla u^n\|_{L^2(\Omega)} \le C N^{-r\alpha/2}\log N, \quad \forall 1\leq n\leq N.
\end{equation}
\end{theorem}
\begin{proof}
According to \eqref{case22}, \eqref{eq:prop_modgm} and Lemma \ref{lem:rkgraded}, we have
\begin{equation*}
\begin{aligned}
& \| \nabla u(t_n) -\nabla u^n\|_{L^2(\Omega)}^2 \le C\sum_{k=1}^n \tau_k^{\alpha} \|r_k\|_{L^2(\Omega)}^2
    \leq C N^{-r \alpha } \sum_{k=1}^n k^{(r-1)
    \alpha} k^{-2\min\{r\alpha,3-\alpha\}}\\ =  & C N^{-r \alpha } \sum_{k=1}^n k^{-\min\{r\alpha+\alpha,6-r\alpha-\alpha\}}.
\end{aligned}
\end{equation*}
Then \eqref{eq:error_modgm} and \eqref{eq:error_modgm2} can be derived, using the techniques in the proof of \cite[Lemma 5.2]{stynes2017error}.
\end{proof}
\begin{figure}[!ht]
    \centering
    \includegraphics[trim={1.5in 0 1.5in 0},clip,width=0.85\textwidth]{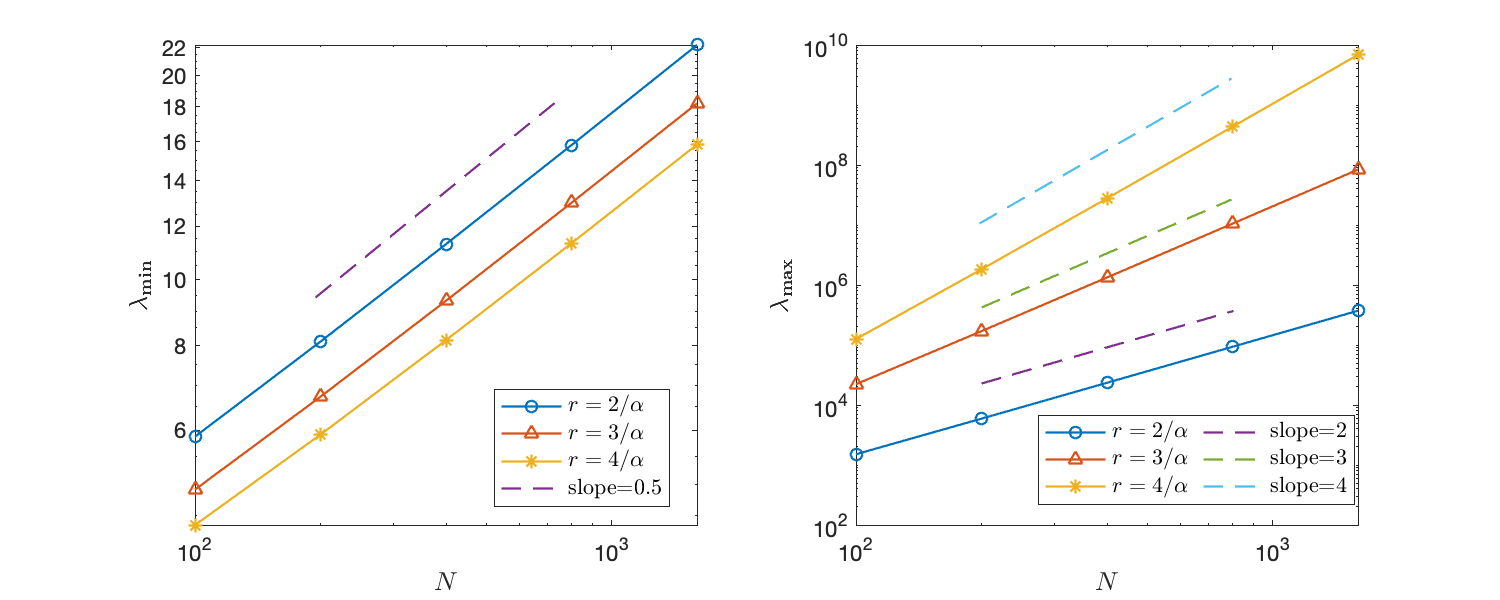}
     \vspace{-0.15in}
    \caption{Minimum and maximum eigenvalues of  the  matrix $\frac{\mathbf M+\mathbf M^{T}}{2}$ for different values of $r$ and $N$, where $T=1$ and $\alpha=0.5$.}
    \label{fig:eigL2}
\end{figure}

\begin{remark}
Based on the positive definiteness result \eqref{eq:corcor} in Corollary \ref{cor}, we can derive lower bounds of the minimum and maximum   eigenvalues for the  matrix $(\mathbf M+\mathbf M^{T})/2$:
\begin{equation*}
\begin{aligned}
    \lambda_{\rm min}=\min_{\|\mathbf v\|_2=1} \mathbf v \mathbf M \mathbf v^T\ge C\tau_{\rm max}^{-\alpha},\quad
    \lambda_{\rm max}=\max_{\|\mathbf v\|_2=1} \mathbf v \mathbf M \mathbf v^T\ge C \tau_{\rm min}^{-\alpha},
    \end{aligned}
\end{equation*}
where $\mathbf v = (v_1,\ldots,v_N)\in \mathbb R^N$, $\tau_{\rm min}=\min_{1\le k\le N}\tau_k$ and $\tau_{\rm max}=\max_{1\le k\le N} \tau_k$.
For the modified graded mesh \eqref{eq:modgm} with grading parameter $r$, according to \eqref{eq:prop_modgm}, 
we have
\begin{equation*}
    \lambda_{\rm min}\ge C N^\alpha \quad \text{and} \quad     \lambda_{\rm max}\ge CN^{r\alpha}.
\end{equation*}
 Figure \ref{fig:eigL2} gives an example of $\lambda_{\rm min}$ and $\lambda_{\rm max}$ for different $r$ and $N$, where it is observed numerically that $\lambda_{\rm min}$ is $\mathcal O(N^\alpha)$ and  $\lambda_{\rm max}$ is $\mathcal O(N^{r\alpha})$. This indicates that the lower bound in \eqref{eq:corcor} is optimal. 
\end{remark}

 \section{Numerical tests}\label{sect6}
 We restrict our testing to the L2 scheme \eqref{eq:sch_sub_graded} on modified graded meshes \eqref{eq:modgm} for simplicity.
Consider the subdiffusion equation \eqref{eq:subdiffusion1} with $T=1$, $\Omega=[-1,1]^2$, and
$
f(t,x,y)=\left(\Gamma(1+\alpha)+2\pi^2 t^\alpha\right)\sin(\pi x)\sin( \pi y).
$
The exact solution to this subdiffusion equation is given by
$
u(t,x,y)=t^\alpha\sin(\pi x)\sin(\pi y).
$
We employ the spectral collocation method \cite{trefethen2000spectral,shen2011spectral} in space with $20^2$ Chebyshev--Gauss--Lobatto points for this example.

Table \ref{tab} displays the maximum $H^1$-errors of the numerical solutions of the L2 scheme for different values of $\alpha$, $r$, and $N$. Theorem \ref{thm:conv_graded} predicts that when $r=2/\alpha$, $3/\alpha-1$, and $5/\alpha-1$, the maximum $H^1$-errors shall be $\mathcal O(N^{-1})$, $\mathcal O(N^{-(3-\alpha)/2})$, and $\mathcal O(N^{-(5-\alpha)/2}\log N)$, respectively. However, as shown in Table \ref{tab}, the practical convergence orders in $H^1$-norm are approximately $\min\{r\alpha,3-\alpha\}$, consistent with the global convergence results in $L^2$-norm proved in \cite{kopteva2021error}. Intuitively, the sharp convergence order in $H^1$-norm may be $\min\{r\alpha,3-\alpha\}$ for modified graded meshes. However, we can only prove the ${(5-\alpha)}/{2}$ order when $r>5/\alpha-1$ at this time.

 \begin{table}[htb!]
 \small
\renewcommand\arraystretch{1.}
\begin{center}
\def\temptablewidth{1\textwidth}
\caption{ Maximum $H^1$-errors and convergence orders for the L2 scheme on the modified graded meshes with different $\alpha$, $r$ and $N$.}\vspace{0in}\label{tab}
{\rule{\temptablewidth}{1pt}}
\begin{tabular*}{\temptablewidth}{@{\extracolsep{\fill}}ccccccc}
 &&$N=200$ & $N=400$&$N=800 $ & $N=1600 $&$N=3200$\\ \hline
$\alpha=0.3$ &$r=\frac{2}{\alpha}$& 4.5156-04&	1.1599-04&	2.9342e-05&	7.3758e-06& 1.8488e-06\\
&&  -- & 1.9609   & 1.9829  &  1.9921   & 1.9962\\
&$r=\frac{3-\alpha}{\alpha}$& 8.3905e-05&	1.3473e-05&	2.1163e-06&	3.2901e-07	&5.0889e-08\\
&&  -- & 2.6387  &  2.6704  &  2.6854   & 2.6927\\
&$r=\frac{5-\alpha}{\alpha}$&1.1086e-05&	1.6360e-06&	2.3780e-07&	3.4502e-08&	5.0261e-09\\
&&  -- &  2.7605   & 2.7823   & 2.7850   & 2.7792

\\
\hline
$\alpha=0.5$ &$r=\frac{2}{\alpha}$& 1.9100e-04&	4.8699e-05&	1.2265e-05&	3.0760e-06&	7.7008e-07\\
&&  -- &  1.9717  &  1.9893 &   1.9955  &  1.9980
\\
&$r=\frac{3-\alpha}{\alpha}$&5.2759e-05&	9.5269e-06&	1.7007e-06&	3.0208e-07&	5.3527e-08\\
&&  -- & 2.4693   & 2.4858 &   2.4931 &   2.4966\\
&$r=\frac{5-\alpha}{\alpha}$& 1.2148e-05&	2.1291e-06&	3.7225e-07&	6.5143e-08&	1.1421e-08\\
&&  -- &     2.5124  &  2.5159   & 2.5146 &   2.5119\\
\hline
$\alpha=0.7$ &$r=\frac{2}{\alpha}$& 1.7692e-04&	4.4591e-05&	1.1183e-05&	2.7980e-06&	6.9966e-07\\
&&  -- &  1.9883 &   1.9954  &  1.9988  &  1.9997\\
&$r=\frac{3-\alpha}{\alpha}$& 5.7060e-05&	1.1817e-05&	2.4153e-06&	4.9159e-07&	9.9914e-08\\
&&  --& 2.2715   & 2.2906 &   2.2967  &  2.2987\\
&$r=\frac{5-\alpha}{\alpha}$&1.4028e-05	&2.8337e-06&	5.7333e-07&	1.1614e-07&	2.3546e-08\\
&&  -- &    2.3075   & 2.3053 &   2.3035 &   2.3023
\\
\end{tabular*}
{\rule{\temptablewidth}{1pt}}
\end{center}
\end{table}

In Figure \ref{fig:H1error}, we plot the $H^1$-errors at final time $t=T$ with respect to $N$ for different values of $\alpha$ and $r$.
Based on the local convergence results established in \cite{kopteva2021error}, the final-time $L^2$-error is expected to decay as $\mathcal O(N^{-(3-\alpha)})$ when $r>3-\alpha$.
We observe that the final-time $H^1$-errors also exhibit a similar decay rate of $\mathcal O(N^{-(3-\alpha)})$ when $r>3-\alpha$.
However, a rigorous analysis of the pointwise error estimates in the $H^1$-norm of L2 schemes is left for future research.

\begin{figure}[!ht]
    \centering
    \includegraphics[trim={1.5in 0 1.5in 0},clip,width=0.95\textwidth]{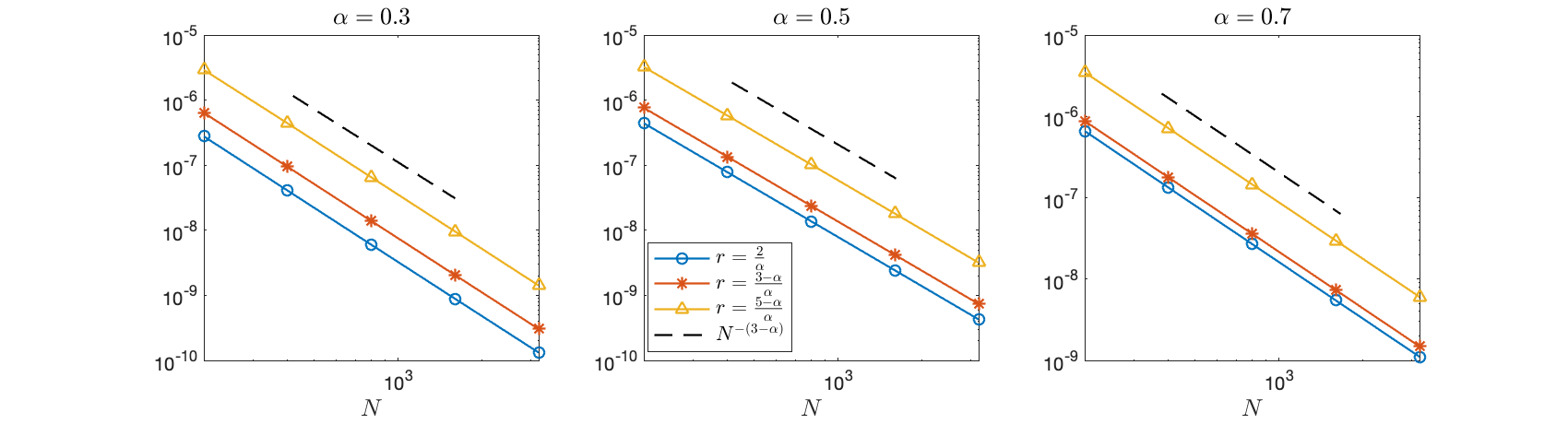}
     \vspace{-0.15in}
    \caption{$H^1$-errors at $t=T$ of numerical solutions of the L2 scheme on the modified graded meshes with different $\alpha$ and $r$.}
    \label{fig:H1error}
\end{figure}

Finally, we test if the positive definiteness result \eqref{eq:corcor} is optimal for this example. 
We define the following ratio used for error analysis in \eqref{case22}:
\begin{equation}
    \gamma_n= \mathcal B_n( e, e)/\sum_{k=1}^n \tau_k^{-\alpha} \|\delta_k e\|^2_{L^2(\Omega)},\quad n\geq 2.
\end{equation}
Fix $N=400$ and $r=(3-\alpha)/\alpha$.
Figure \ref{fig:gamman} plots $\gamma_n$ with respect to $n$ for different values of $\alpha$. The quantity $\gamma_n$ is observed to have lower and upper bounds.
Therefore, the positive definiteness result \eqref{eq:corcor} is optimal. 
 \begin{figure}[!ht]
    \centering
    \includegraphics[trim={1.5in 0 1.5in 0},clip,width=0.95\textwidth]{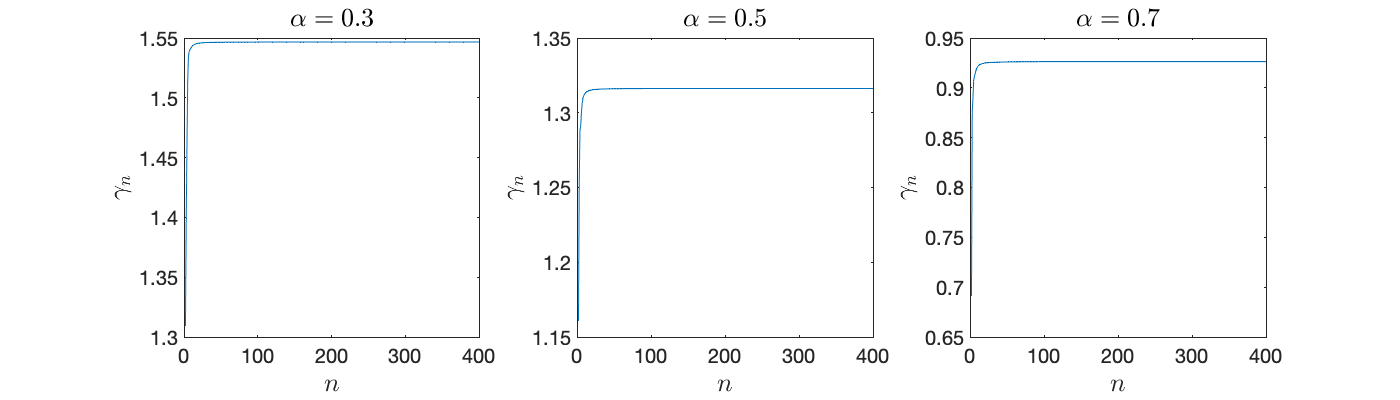}
     \vspace{-0.15in}
    \caption{The values of $\gamma_n$ for $2\le n\le N$  for $\alpha=0.3,\ 0.5,\ 0.7$ where $N=400$ and $r=(3-\alpha)/\alpha$.}
    \label{fig:gamman}
\end{figure}

\section*{Acknowledgements}
The authors thank one anonymous reviewer for suggesting to add Remark \ref{rem:3.4} and another anonymous reviewer for inspiring us to consider the $H^1$-norm error analysis for completeness. 

\bibliography{bibfile}

\providecommand{\bysame}{\leavevmode\hbox to3em{\hrulefill}\thinspace}
\providecommand{\MR}{\relax\ifhmode\unskip\space\fi MR }
\providecommand{\MRhref}[2]{%
  \href{http://www.ams.org/mathscinet-getitem?mr=#1}{#2}
}
\providecommand{\href}[2]{#2}
\begin{thebibliography}{10}

\bibitem{alikhanov2015new}
Anatoly~A Alikhanov, \emph{A new difference scheme for the time fractional
  diffusion equation}, Journal of Computational Physics \textbf{280} (2015),
  424--438.

\bibitem{alikhanov2021high}
Anatoly~A Alikhanov and Chengming Huang, \emph{A high-order {L}2 type
  difference scheme for the time-fractional diffusion equation}, Applied
  Mathematics and Computation \textbf{411} (2021), 126545.

\bibitem{chen2019error}
Hu~Chen and Martin Stynes, \emph{Error analysis of a second-order method on
  fitted meshes for a time-fractional diffusion problem}, Journal of Scientific
  Computing \textbf{79} (2019), no.~1, 624--647.

\bibitem{gao2014new}
Guang-hua Gao, Zhi-zhong Sun, and Hong-wei Zhang, \emph{A new fractional
  numerical differentiation formula to approximate the {Caputo} fractional
  derivative and its applications}, Journal of Computational Physics
  \textbf{259} (2014), 33--50.

\bibitem{gorenflo2002time}
Rudolf Gorenflo, Francesco Mainardi, Daniele Moretti, and Paolo Paradisi,
  \emph{Time fractional diffusion: a discrete random walk approach}, Nonlinear
  Dynamics \textbf{29} (2002), no.~1, 129--143.

\bibitem{huang20200optimal}
Chaobao Huang and Martin Stynes, \emph{{Optimal $H^1$-spatial convergence of a
  fully discrete finite element method for the time-fractional Allen--Cahn
  equation}}, Advances in Computational Mathematics \textbf{46} (2020), no.~4,
  63.

\bibitem{huang2020optimal}
\bysame, \emph{{Optimal spatial $H^1$-norm analysis of a finite element method
  for a time-fractional diffusion equation}}, Journal of Computational and
  Applied Mathematics \textbf{367} (2020), 112435.

\bibitem{huang2022sharp}
\bysame, \emph{{A sharp $\alpha$-robust $L^{\infty}(H^1)$ error bound for a
  time-fractional Allen-Cahn problem discretised by the Alikhanov L2-1$_\sigma$
  scheme and a standard FEM}}, Journal of Scientific Computing \textbf{91}
  (2022), no.~2, 43.

\bibitem{ji2020adaptive}
Bingquan Ji, Hong-lin Liao, Yuezheng Gong, and Luming Zhang, \emph{Adaptive
  linear second-order energy stable schemes for time-fractional {Allen-Cahn}
  equation with volume constraint}, Communications in Nonlinear Science and
  Numerical Simulation \textbf{90} (2020), 105366.

\bibitem{jin2016two}
Bangti Jin, Raytcho Lazarov, and Zhi Zhou, \emph{Two fully discrete schemes for
  fractional diffusion and diffusion-wave equations with nonsmooth data}, SIAM
  Journal on Scientific Computing \textbf{38} (2016), no.~1, A146--A170.

\bibitem{jin2017correction}
Bangti Jin, Buyang Li, and Zhi Zhou, \emph{Correction of high-order {BDF}
  convolution quadrature for fractional evolution equations}, SIAM Journal on
  Scientific Computing \textbf{39} (2017), no.~6, A3129--A3152.

\bibitem{jin2020subdiffusion}
\bysame, \emph{Subdiffusion with time-dependent coefficients: improved
  regularity and second-order time stepping}, Numerische Mathematik
  \textbf{145} (2020), no.~4, 883--913.

\bibitem{kopteva2019error}
Natalia Kopteva, \emph{Error analysis of the {L}1 method on graded and uniform
  meshes for a fractional-derivative problem in two and three dimensions},
  Mathematics of Computation \textbf{88} (2019), no.~319, 2135--2155.

\bibitem{kopteva2021error}
\bysame, \emph{Error analysis of an {L}2-type method on graded meshes for a
  fractional-order parabolic problem}, Mathematics of Computation \textbf{90}
  (2021), no.~327, 19--40.

\bibitem{kopteva2020error}
Natalia Kopteva and Xiangyun Meng, \emph{Error analysis for a
  fractional-derivative parabolic problem on quasi-graded meshes using barrier
  functions}, SIAM Journal on Numerical Analysis \textbf{58} (2020), no.~2,
  1217--1238.

\bibitem{langlands2005accuracy}
TAM Langlands and Bruce~I Henry, \emph{The accuracy and stability of an
  implicit solution method for the fractional diffusion equation}, Journal of
  Computational Physics \textbf{205} (2005), no.~2, 719--736.

\bibitem{liao2018sharp}
Hong-lin Liao, Dongfang Li, and Jiwei Zhang, \emph{Sharp error estimate of the
  nonuniform l1 formula for linear reaction-subdiffusion equations}, SIAM
  Journal on Numerical Analysis \textbf{56} (2018), no.~2, 1112--1133.

\bibitem{liao2019discrete}
Hong-lin Liao, William McLean, and Jiwei Zhang, \emph{A discrete gronwall
  inequality with applications to numerical schemes for subdiffusion problems},
  SIAM Journal on Numerical Analysis \textbf{57} (2019), no.~1, 218--237.

\bibitem{liao2018second}
Hong-Lin Liao, William McLean, and Jiwei Zhang, \emph{A second-order scheme
  with nonuniform time steps for a linear reaction-sudiffusion problem},
  Communications in Computational Physics \textbf{30} (2021), no.~2, 567--601.

\bibitem{liao2020second}
Hong-lin Liao, Tao Tang, and Tao Zhou, \emph{A second-order and nonuniform
  time-stepping maximum-principle preserving scheme for time-fractional
  {Allen-Cahn} equations}, Journal of Computational Physics \textbf{414}
  (2020), 109473.

\bibitem{liao2021analysis}
Hong-lin Liao and Zhimin Zhang, \emph{Analysis of adaptive {BDF}2 scheme for
  diffusion equations}, Mathematics of Computation \textbf{90} (2021), no.~329,
  1207--1226.

\bibitem{lin2007finite}
Yumin Lin and Chuanju Xu, \emph{Finite difference/spectral approximations for
  the time-fractional diffusion equation}, Journal of Computational Physics
  \textbf{225} (2007), no.~2, 1533--1552.

\bibitem{lv2016error}
Chunwan Lv and Chuanju Xu, \emph{Error analysis of a high order method for
  time-fractional diffusion equations}, SIAM Journal on Scientific Computing
  \textbf{38} (2016), no.~5, A2699--A2724.

\bibitem{metzler2000random}
Ralf Metzler and Joseph Klafter, \emph{The random walk's guide to anomalous
  diffusion: a fractional dynamics approach}, Physics reports \textbf{339}
  (2000), no.~1, 1--77.

\bibitem{CSIAM-AM-1-478}
Chaoyu Quan, Tao Tang, and Jiang Yang, \emph{How to define
  dissipation-preserving energy for time-fractional phase-field equations},
  CSIAM Transactions on Applied Mathematics \textbf{1} (2020), no.~3, 478--490.

\bibitem{quan2020numerical}
Chaoyu Quan, Tao Tang, and Jiang Yang, \emph{Numerical energy dissipation for
  time-fractional phase-field equations}, arXiv preprint arXiv:2009.06178
  (2020).

\bibitem{quan2022energy}
Chaoyu Quan and Boyi Wang, \emph{Energy stable {L}2 schemes for time-fractional
  phase-field equations}, Journal of Computational Physics \textbf{458} (2022),
  111085.

\bibitem{shen2011spectral}
Jie Shen, Tao Tang, and Li-Lian Wang, \emph{Spectral methods: algorithms,
  analysis and applications}, vol.~41, Springer Science \& Business Media,
  2011.

\bibitem{stynes2017error}
Martin Stynes, Eugene O'Riordan, and Jos{\'e}~Luis Gracia, \emph{Error analysis
  of a finite difference method on graded meshes for a time-fractional
  diffusion equation}, SIAM Journal on Numerical Analysis \textbf{55} (2017),
  no.~2, 1057--1079.

\bibitem{sun2006fully}
Zhi-zhong Sun and Xiaonan Wu, \emph{A fully discrete difference scheme for a
  diffusion-wave system}, Applied Numerical Mathematics \textbf{56} (2006),
  no.~2, 193--209.

\bibitem{trefethen2000spectral}
Lloyd~N Trefethen, \emph{Spectral methods in {MATLAB}}, SIAM, 2000.

\end{thebibliography}
\bibliographystyle{amsplain}

\end{document}